\documentclass{interact}
\usepackage{epstopdf}
\usepackage[caption=false]{subfig}
\usepackage{booktabs}   

\usepackage[numbers,sort&compress]{natbib}
\bibpunct[, ]{[}{]}{,}{n}{,}{,}
\renewcommand\bibfont{\fontsize{10}{12}\selectfont}
\makeatletter
\def\NAT@def@citea{\def\@citea{\NAT@separator}}
\makeatother

\usepackage{a4wide,latexsym,varioref,amsmath,xcolor,bbm,amsfonts}
\definecolor{blue(pigment)}{rgb}{0.2, 0.2, 0.6}

\newcommand{\pagenumbers}{\arabic{page}}
\renewcommand{\thepage}{\pagenumbers}
\newcommand{\numsection}[1]{\section{#1}\setcounter{equation}{0}}

\newcommand{\beqn}[1]{\begin{equation}\label{#1}}
\newcommand{\eeqn}{\end{equation}}

\newcommand{\tim}[1]{\;\; \mbox{#1} \;\;}

\theoremstyle{plain}
\newtheorem{theorem}{Theorem}[section]

\newtheorem{lemma}[theorem]{Lemma}
\newcommand{\llem}[2]{\vspace{\baselineskip} 
\noindent\framebox[\textwidth]{\parbox{0.95\textwidth}{
\begin{lemma} \label{#1} \rm #2 \end{lemma} } } \vspace{\baselineskip} }

\newlength{\thmw}
\newcommand{\lbthm}[3]{\vspace{\baselineskip}\noindent\hbox{%
  \lower\fboxrule\hbox{\vbox{\hrule\hbox{\vrule \kern-\fboxrule \vbox{%
  \vspace{\fboxsep} \noindent\hspace{2\fboxsep}\parbox{\thmw}{
  \begin{theorem}\label{#1}{\rm #2}\end{theorem}\vspace{-\lastskip}}
  \hspace{\fboxsep}}\kern-\fboxrule \vrule }}}}\newpage \hbox{%
  \lower\fboxrule\hbox{\vbox{\hbox{\vrule \kern-\fboxrule \vbox{%
  \noindent\hspace{2\fboxsep}\parbox{\thmw}{\rm #3}\hspace{\fboxsep}
  \vspace{4\fboxsep}}\kern-\fboxrule \vrule }\hrule }}}\vspace{\baselineskip}
}
\newcounter{algo}[section]
\renewcommand{\thealgo}{\thesection.\arabic{algo}}
\newcommand{\algo}[3]{\refstepcounter{algo}
\begin{center}\begin{figure}[htpb]
\framebox[\textwidth]{
\parbox{0.95\textwidth} {\vspace{\topsep}
{\bf Algorithm \thealgo : #2}\label{#1}\\
\vspace*{-\topsep} \mbox{ }\\
{#3} \vspace{\topsep} }}
\end{figure}\end{center}}

\newcommand{\calD}{{\cal D}}

\renewcommand{\Re}{\hbox{I\hskip -2pt R}}

\newcommand{\eqdef}{\stackrel{\rm def}{=}}

\newcommand{\ii}[1]{\{1, \ldots, #1 \}}

\theoremstyle{defn}
\newtheorem{defn}{Definition}[section]
\theoremstyle{assumption}
\newtheorem{assumption}{Assumption}[section]
\theoremstyle{remark}

\begin{document}
\title{Stochastic Analysis of an Adaptive Cubic Regularisation Method under Inexact Gradient Evaluations and Dynamic Hessian Accuracy }

\author{\name{Stefania Bellavia\textsuperscript{a}\thanks{CONTACT: Stefania Bellavia, Email: stefania.bellavia@unifi.it}
and Gianmarco Gurioli\textsuperscript{b}}\affil{\textsuperscript{a}
Dipartimento di Ingegneria Industriale,
  Universit\`{a} degli Studi, Firenze, Italy; \textsuperscript{b} Dipartimento di Matematica e Informatica ``Ulisse Dini",
  Universit\`{a} degli Studi, Firenze, Italy.}
}



\maketitle{}

\begin{abstract}
We here adapt an extended version of the adaptive cubic regularisation method with dynamic inexact Hessian information for nonconvex optimisation in \cite{IMA} to the stochastic optimisation setting. While exact function evaluations are still considered, this novel variant inherits the innovative use of adaptive accuracy requirements for Hessian approximations introduced in \cite{IMA} and additionally employs inexact computations of the gradient. Without restrictions on the variance of the errors, we assume that these approximations are available within a sufficiently large, but fixed, probability and we extend, in the spirit of \cite{CartSche17}, the deterministic analysis of the framework to its stochastic counterpart, showing that the expected number of iterations to reach a first-order stationary point matches the well known worst-case optimal complexity. This is, in fact, still given by $O(\epsilon^{-3/2})$, with respect to the first-order $\epsilon$ tolerance. Finally, numerical tests on nonconvex finite-sum minimisation confirm that using inexact first and second-order derivatives can be  beneficial  in terms of the computational savings.
\end{abstract}

\begin{keywords}
Adaptive cubic regularization methods; inexact derivatives evaluations; stochastic nonconvex optimization; worst-case complexity analysis; finite-sum minimization.
\end{keywords}

\numsection{Introduction} Adaptive Cubic Regularisation (ARC) methods are Newton-type procedures for solving unconstrained optimisation problems of the form
\begin{eqnarray}
\min_{x \in \mathbb{R}^n}  f(x), \label{problem1}
\end{eqnarray}
in which $f:\mathbb{R}^n\rightarrow \mathbb{R}$ is a sufficiently smooth, bounded below and, possibly, nonconvex function. In the seminal work by \cite{NP} the iterative scheme of the method is based on the minimisation of a cubic model, relying on the Taylor series, for predicting the objective function values, and is a globally convergent second-order procedure. The main reason to consider the ARC framework in place of other globalisation strategies, such as Newton-type methods embedded into a linesearch or a trust-region scheme, lies on its optimal complexity. In fact, given the first-order $\epsilon$ tolerance and assuming Lipschitz continuity of the Hessian of the  objective function,  an $\epsilon$-approximate first-order stationary point is reached, in the worst-case, in at most $O(\epsilon^{-3/2})$ iterations, instead of the  $O(\epsilon^{-2})$ bound gained by trust-region and linesearch methods \cite{CGToint, CGT}. More in depth, 
an $(\epsilon,\epsilon_H)$-approximate first- and second-order critical point is found in at most 
$O(\max(\epsilon^{-3/2}, \epsilon_H^{-3}))$ iterations, where $\epsilon_H$ is the positive  prefixed  second-order optimality tolerance \cite{ ARC2, CGToint,CGTIMA, NP}. 
We observe that, in \cite{carmon_optcompl} it has been shown that the bound $O(\epsilon^{-3/2})$ for computing an $\epsilon$-approximate first-order stationary point is optimal among methods operating on functions with Lipschitz continuous Hessian.
Experimentally, second-order methods can be  more efficient than first-order ones on badly scaled and ill-conditioned problems, since they  take advantage of curvature information to easily escape from saddle points to search for local minima (\cite{review, CGToint,Roosta_2p}) and this feature is in practice quite robust to the use of inexact Hessian information. On the other hand, their per-iteration cost is expected to be higher than first-order procedures, due to the computation of the Hessian-vector products. Consequently, literature has recently focused on ARC variants with inexact  derivative information, starting from schemes employing Hessian approximations \cite{IMA, Cin, Roosta}
 though conserving optimal complexity.  ARC methods with inexact gradient and Hessian approximations and still preserving optimal complexity are given in  \cite{ BellGuriMoriToin19, CartSche17, kl,  Roosta_2p, Roosta_inexact, zhou_xu_gu}.
These approaches have mostly been applied to large-scale finite-sum minimisation problems
\begin{equation}
\label{finite-sum}
\min_{x\in\mathbb{R}^n} f(x)=\frac{1}{N}\sum_{i=1}^N{\varphi_i(x)},
\end{equation}
widely used in machine learning applications. In this setting, the objective function $f$ is the mean of $N$ component functions $\varphi_i:\mathbb{R}^n\rightarrow \mathbb{R}$ and, hence, the evaluation of the exact derivatives might be, for larger values of $N$, computationally expensive. In the papers cited above the derivatives approximations are required to fulfil given accuracy requirements and are computed by random sampling. The size of the sample is determined as to satisfy the prescribed accuracy with a sufficiently large prefixed probability
exploiting  the operator Bernstein inequality for tensors (see \cite{Tropp}). To deal with the nondeterministic aspects of these algorithms,  in \cite{CartSche17,zhou_xu_gu} probabilistic models are considered and it is proved that, in expectation, optimal complexity applies as in the deterministic case; 
 in \cite{IMA, BellGuriMoriToin19, Cin, Roosta, Roosta_inexact} high probability results are given and it is shown that the optimal complexity result  is restored in probability. Nevertheless, this latter analysis does not provide information on the behaviour of the method when the desired accuracy levels in derivatives approximations  are not fulfilled.
 With the aim of filling this gap, we here perform the stochastic analysis of  the framework in \cite{IMA}, where approximated Hessians are employed.
 To make the method  more general, inexactness is allowed  in   first-order information, too. The analysis aims at bounding the expected number of iterations required by the algorithm to reach a first-order stationary point, under the assumption that gradient and Hessian approximations are available within a sufficiently large, but fixed, probability, recovering optimal complexity in the spirit of \cite{CartSche17}. 

The rest of the paper is organised as follows.  In section 1.1 we briefly survey the related works and in section 1.2 we summarise our contributions.  In Section 2 we introduce a stochastic ARC algorithm with inexact gradients and dynamic Hessian accuracy and state the main assumptions on the stochastic process induced by the algorithm. Relying on several existing results and deriving some additional outcomes, Section 3 is then devoted to perform the complexity analysis of the framework, while Section 4 proposes a practical guideline to apply the method for solving finite-sum minimisation problems. Numerical results for nonconvex finite-sum minimisation problems are discussed in Section 5 and concluding remarks are finally given in Section 6.
\vskip 5pt
\noindent
{\bf Notations.} The Euclidean vector and matrix norm is denoted as $\|\cdot \|$.
Given the scalar or vector or matrix $v$, and the non-negative scalar $\chi$, 
we write $v=O(\chi)$ if there is a constant $g$ such that $\|v\| \le  g \chi$.
Given any set ${\cal{S}}$, $|{\cal{S}}|$ denotes its cardinality. As usual, $\mathbb{R}^+$ denotes the set of positive real numbers.

\subsection{Related works}
The interest in ARC methods with inexact derivatives has been  steadily increasing.  We are here interested in computable  accuracy requirements for gradient and  Hessian approximations, preserving optimal complexity of these procedures. 
Focusing on the Hessian approximation, in \cite{ARC2} it has been proved that optimal complexity is conserved provided that, at each iteration $k$,   the Hessian approximation $\overline{\nabla^2 f}(x_k)$ satisfies 
\begin{equation}
\label{condKL}
\|(\overline{\nabla^2 f}(x_k)-\nabla^2 f(x_k))s_k\|\le \chi \|s_k\|^2,
\end{equation}
where  ${\nabla^2 f}(x_k)$ denotes the true Hessian at $x_k$.
The method   in \cite{kl}, specifically designed to minimise finite-sum problems,  assume that  
$\overline{\nabla^2 f}(x_k)$   satisfies 
\begin{equation}\label{kl2}
\| \overline{\nabla^2 f}(x_k) -\nabla^2 f(x_k)\|\le \chi\|s_k\|
\end{equation}
 with $\chi$ a positive constant, leading to 
 \eqref{condKL}.  
 Unfortunately, the upper bound in use depends on the steplength $\|s_k\|$ which is unknown when forming the Hessian approximation $ \overline{\nabla^2 f}(x_k)$. 
Finite-differences versions  of ARC method have been   investigated in \cite{CGTfinitedifference}.  The Hessian approximation satisfies  \eqref{kl2} and its computation  requires an inner-loop to meet the accuracy requirement.
  This mismatch is circumvented, in practical implementations of the method in \cite{kl}, by taking the step length at the previous iteration. Hence, this approach is unreliable when the norm of the step varies significantly from an iteration to the other, as also noticed in the numerical tests of \cite{IMA}. To overcome this practical issue, Xu and others replace in \cite{Roosta} the 
  accuracy requirement  \eqref{kl2} with 
 \begin{equation}\label{H_roosta}
\| \overline{\nabla^2 f}(x_k) -\nabla^2 f(x_k)\|\le \chi\epsilon, 
\end{equation}
where $\epsilon$  is  the first-order  tolerance. This
provides them with $\|(\overline{\nabla^2 f}(x_k)-\nabla^2 f(x_k))s_k\|\le \chi \epsilon \|s_k\|$, used to prove optimal complexity. In this situation, the estimate $\overline{\nabla^2 f}(x_k)$ is practically computable, independently of the step length, but at the cost of a very restrictive accuracy requirement (it is defined in terms of the $\epsilon$ tolerance) to fulfil at each iteration of the method. 
We further note that, in \cite{Wang}, optimal complexity results for a cubic regularisation method employing the implementable  condition 
\begin{equation}
\|\overline{\nabla^2 f}(x_k)-\nabla^2 f(x_k)\|\le \chi \|s_{k-1}\|
\end{equation}
are given under the assumption that the constant regularisation parameter is greater than the Hessian Lipschitz constant. Then, the knowledge of the Lipschitz constant is assumed.  Such an assumption can be quite stringent, especially when minimising nonconvex objective functions. On the contrary, adaptive cubic regularisation frameworks get rid of the Lipschitz constant, trying to overestimate it by an adaptive procedure that is well defined provided that the approximated Hessian is accurate enough.
To our knowledge, accuracy requirements  depending on the current step, as those in \eqref{condKL}-\eqref{H_roosta}, are needed to prove that the step acceptance criterion is well-defined and the regularisation parameter is bounded above.

\noindent
Regarding the gradient approximation,  the accuracy requirement in \cite{CGTfinitedifference, kl}  has the following form
\begin{equation}
\label{gradient}
\|\overline{\nabla f}(x_k)-\nabla f(x_k)\|\le \mu \|s_k\|^2,
\end{equation}
where $\overline{\nabla f}(x_k)$ denotes the gradient approximation and  $\mu$ is a positive constant. Then, the accuracy requirement depends on the norm of the step again.

\noindent
In \cite{Roosta_inexact}, as for the Hessian approximation, in order to get rid of the norm of the step,
a  very  tight accuracy requirement in used as the absolute error has to be of the order of $\epsilon^2$ at each iteration, 
i.e.
\begin{equation}
\label{gradient_roosta}
\|\overline{\nabla f}(x_k)-\nabla f(x_k)\|\le \mu \epsilon^2.
\end{equation}
As already noticed, in  \cite{Roosta, Roosta_inexact}, a complexity analysis in high probability is carried out in order to cover the situation where  accuracy requirements \eqref{H_roosta} and \eqref{gradient_roosta} are satisfied only with a sufficiently large probability. While the behaviour of cubic regularisation approaches employing approximated  derivatives is analysed in expectation in
 \cite{CartSche17}, assuming that \eqref{condKL} and \eqref{gradient} are satisfied with high probability.
\noindent
In the finite-sum minimisation context, accuracy requirements   \eqref{condKL}, \eqref{kl2} and \eqref{gradient}  can be enforced with high probability by subsampling 
 via an inner iterative process. Namely, the approximated derivative  is computed using
 a predicted accuracy, the step $s_k$ is computed and, if the predicted accuracy is larger than the required accuracy, the predicted accuracy is progressively decreased (and the sample size 
 progressively  increased) until  the accuracy requirement is satisfied. 

\noindent
 The cubic regularisation variant proposed in  \cite{IMA} employs exact gradient and ensures condition \eqref{condKL}, avoiding the above vicious cycle, requiring that 
\begin{equation}
\label{AccDynH}
\|\overline{\nabla^2 f}(x_k)-\nabla^2 f(x_k)\|\le c_k,
\end{equation}
where the guideline for choosing $c_k$ is as follows:
\begin{equation} \label{ck}
c_k\le\left \{ \begin{array}{ll}
 c, \quad c>0,\qquad \qquad  ~~~ \mbox{if } \ \ \|s_k\|\ge 1, \\
  \alpha(1-\beta)\| \nabla f(x_k)\|, \quad    \mbox{if } \ \ \|s_k\|< 1,
\end{array}
\right.
\end{equation}
with $0\le \alpha<\frac23$ and $0<\beta<1$.
Note that, for a sufficiently large constant $c$, the accuracy requirement $c_k$ can be less stringent than $\epsilon$ when $\|s_k\|\ge 1$ or, otherwise, as long as $\alpha(1-\beta)\|\nabla f(x_k)\|>\epsilon$.
Despite condition \eqref{ck} still involves the norm of the step,  the accuracy requirement \eqref{AccDynH} can be implemented without requiring an inner loop (see, \cite{IMA} and Algorithm \ref{algo}). 

\noindent
We finally  mention  that regularisation methods employing inexact
derivatives  and also inexact function values are proposed in \cite{BellGuriMoriToin19} and the complexity analysis carried out  covers arbitrary optimality order and arbitrary degree of the available approximate derivatives. Also in this latter approach, the accuracy requirement in derivatives  approximation  depends on the norm of the step and  an inner loop is needed  in order to increase the accuracy and meet the accuracy requirements. A different approach based on the Inexact Restoration framework is given  in \cite{bkm} where, in the context of finite-sums problems, the sample size rather than the 
approximation accuracy is adaptively chosen.

\subsection{Contributions}

In light of the related works the main contributions of this paper are the following:
\vspace{0.1cm}
\begin{itemize}
\item   We generalise the method given in \cite{IMA}. In particular, we kept   the practical adaptive criterion \eqref{AccDynH}, which is implemented without including an inner loop, allowing  inexactness in the gradient as well.  Namely, inspired by \cite{BellGuriMoriToin19}, we require that the gradient approximation satisfies the following relative implicit  condition:
\begin{equation}
\label{gradient_nostro}
\|\overline{\nabla f}(x_k)-\nabla f(x_k)\|\le \zeta_k \|\overline \nabla f(x_k)\|^2,
\end{equation}
where $\zeta_k$ is an iteration-dependent nonnegative parameter.
Unlike \cite{CartSche17} and \cite{kl} (see \eqref{gradient}), this latter condition does not depend on the norm of the step. Thus, its practical implementation calls for an inner loop that can be   performed before the step computation and extra-computations of the step are not needed.
A detailed description of a practical implementation of this accuracy requirement in subsampling scheme for finite-sum minimisation is given  in Section 4.
\vspace{0.1cm}
\item  We assume that the accuracy requirements \eqref{AccDynH} and \eqref{gradient_nostro} are satisfied with high probability and we perform, in the spirit of \cite{CartSche17}, the stochastic analysis of the resulting method, showing that 
the expected number of iterations needed to reach an $\epsilon$-approximate first-order critical point  is, in the worst-case, of the order of 
$\epsilon^{-3/2}$. This analysis also applies to the method given in \cite{IMA}.
\end{itemize}

\numsection{A stochastic cubic regularisation algorithm with inexact derivatives evaluations}
Before introducing our stochastic algorithm, we consider the following hypotheses on $f$.\\

\begin{assumption}
\label{Assf}
With reference to  problem \eqref{problem1}, the objective function $f$ is assumed to be:
\begin{itemize}
\item[(i)] bounded below by $f_{low}$, for all $x\in\Re^n$; \vskip 5pt 
\item[(ii)] twice continuously differentiable, i.e. $f\in\mathcal{C}^2(\mathbb{R}^n)$; \vskip 5pt 
\end{itemize}
Moreover,\vskip 5pt 
\begin{itemize}
\item[(iii)]  the Hessian is globally Lipschitz continuous with Lipschitz constant $L_H>0$,  i.e.,
\begin{equation}\label{LipHess}
\qquad \|\nabla^2 f(x)-\nabla^2 f(y)\| \le L_H\| x-y\|,
\end{equation}
for all $x$, $y\in\mathbb{R}^n$.
\end{itemize}
\end{assumption}
\vspace{0.2cm}
\noindent
The iterative method we are going to introduce is, basically, the stochastic counterpart of an extension of the one proposed in \cite{IMA}, based on first and second-order inexact information. More in depth at iteration $k$, given the trial step $s$, the value of the objective function at $x_k+s$ is predicted by mean of a cubic model $m_k(x_k,s,\sigma_k)$ defined in terms of an approximate Taylor expansion of $f$ centered at $x_k$ with increment $s$, truncated to the second order, namely
\begin{equation}
\label{m}
m_k(x_k,s,\sigma_k)= f(x_k)+\overline{ \nabla f}(x_k)^T s+\frac{1}{2} s^T \overline{ \nabla^2 f}(x_k) s+\frac{\sigma_k}{3}\|s\|^3\eqdef  \overline T_2(x_k,s)+\frac{\sigma_k}{3}\|s\|^3,
\end{equation}
in which both the gradient $\overline{ \nabla f}(x_k)$ and the Hessian matrix $ \overline{ \nabla^2 f}(x_k)$ represent approximations of $ \nabla f(x_k)$ and  $\nabla^2 f(x_k)$, respectively. According to the basic ARC framework in \cite{ARC1}, the main idea is to approximately minimise, at each iteration, the cubic model and to adaptively search for a regulariser $\sigma_k$ such that the following overestimation property is satisfied:
\[
f(x_k+s)\le m_k(x_k,s,\sigma_k),
\]
in which $s$ denotes the approximate minimiser of $m_k(x_k,s,\sigma_k)$. Within these requirements, it follows that
\[
f(x_k)=m_k(x_k,0,\sigma_k)\ge m_k(x_k,s,\sigma_k)\ge f(x_k+s),
\] 
so that the objective function is not increased when moving from $x_k$ to $x_k+s$. To get more insight, the cubic model \eqref{m} is approximately minimised in the sense that the minimiser $s_k$ satisfies
\begin{eqnarray}
& & m_k(x_k,s_k,\sigma_k)<m_k(x_k,0,\sigma_k),\label{mdecr} \\
& & \| \nabla_s m_k(x_k,s_k,\sigma_k)\| \le \beta_k \|\overline{\nabla f}(x_k)\|, \label{tc}  
\end{eqnarray}
for all $k\ge 0$ and some $0\le \beta_k\le \beta$, $\beta \in [0,1)$. Practical choices for $\beta_k$ are, for instance, $\beta_k= \beta \min \left( 1, \frac{\| s_k\|^2}{\|\overline{\nabla f} (x_k)\|} \right)$ or $\beta_k= \beta \min(1,\|s_k\|)$ (see, e.g., \cite{IMA}), leading to
\begin{equation}
\label{tcsub}
\| \nabla_s m_k(x_k,s_k,\sigma_k)\| \le \beta \min \left( \| s_k\|^2, \|\overline{\nabla f} (x_k)\|  \right), 
\end{equation}
and
\begin{equation}
\label{tc.s}
 \| \nabla_s m_k(x_k,s_k,\sigma_k)\| \le \beta \min(1,\|s_k\|) \|\overline{\nabla f} (x_k)\|,
\end{equation}
respectively. We notice that, if the overestimation property $f(x_k+s)\le m_k(x_k,s,\sigma_k)$ is satisfied, the requirement \eqref{mdecr} implies that $f(x_k)=m_k(x_k,0,\sigma_k)> m_k(x_k,s,\sigma_k)\ge f(x_k+s)$, resulting in a decrease of the objective. The trial point $x_k+s_k$ is then used to compute the relative decrease \cite{Toint1}
 \beqn{rhokdef2}
 \rho_k = \frac{f(x_k) - f(x_k+s_k)}
               {\overline T_2(x_k,0)-\overline T_2(x_k,s_k)}.
 \eeqn
 If $\rho_k\ge \eta$, with $\eta\in(0,1)$ a prescribed decrease fraction, then the trial point is accepted, the iteration is declared successful, the regularisation parameter is decreased by a factor $\gamma$ and we go on recomputing the approximate model at the updated iterate; otherwise, an unsuccessful iteration occurs: the point $x_k+s_k$ is rejected, the regulariser is increased by a factor $\gamma$, a new approximate model at $x_k$ is computed and a new trial step $s_k$ is recomputed. 
At each iteration, the model $m_k(x_k,s,\sigma_k)$  involved  relies on inexact quantities, that can be considered as realisations of random variables. Hereafter, all random quantities are denoted by capital letters, while the use of small letters is reserved for their realisations. In particular, let us denote a random model at iteration $k$ as $M_k$, while we use the notation $m_k=M_k(\omega)$ for its realisation, with $\omega$ a random sample taken from a context-dependent probability space $\Omega$.  In particular, 
 we denote by $\overline{\nabla f}(X_k)$ and $\overline{\nabla^2 f}(X_k)$ the random variables for $\overline{\nabla f}(x_k)$ and $\overline{\nabla^2 f}(x_k)$, respectively. Consequently, the iterates $X_k$, as well as the regularisers $\Sigma_k$ and the steps $S_k$ are the random variables such that $x_k=X_k(\omega)$, $\sigma_k=\Sigma_k(\omega)$ and $s_k=S_k(\omega)$. 
 
 \noindent
 The focus of this paper is to derive the expected worst-case complexity bound to approach a first-order optimality point, that is, given a tolerance $\epsilon\in(0,1)$, the number of steps $\overline{k}$ (in the worst-case) such that an iterate $x_{\overline{k}}$ satisfying 
\begin{eqnarray}
 \|\nabla f(x_{\overline{k}})\|\le \epsilon\nonumber
 \end{eqnarray} 
is reached. 
To this purpose, after the description of the algorithm, we state the main definitions and hypotheses needed to carry on with the analysis up to the complexity result. Our algorithm is reported below. 

\algo{algo}{Stochastic ARC algorithm with inexact gradient and dynamic Hessian accuracy}
{\vspace*{-0.3 cm}
\begin{description}
\item[Step 0: Initialisation.]
  An initial point $x_0\in\mathbb{R}^n$  and an initial regularisation parameter $\sigma_0>0$
  are given. The constants $\beta$, $\alpha$,  $\eta$, $\gamma$, $\sigma_{\textrm{min}}$ and $c$ are also given such that
\begin{eqnarray}
 0<\beta<1, ~ \alpha\in  \left[0, \displaystyle \frac 2 3\right), ~\sigma_{\min}\in (0, \sigma_0],~  0<\eta < \frac{2-3\alpha}{2}, ~\gamma>1,~c>0.\label{initialconsts}
\end{eqnarray}
  Compute $f(x_0)$ and set $k=0$,  ${\rm flag}=1$.

\vspace{2mm}
 \item[Step 1: Gradient approximation. ] Compute an approximate gradient $\overline{\nabla f}(x_k)$ 
\vspace{2mm}
\item[Step 2: Hessian approximation (model costruction). ] 
If ${\rm flag}=1$ set  $c_{k}=c$, else set $c_{k}=\alpha(1-\beta)\|\overline{\nabla f}(x_{k})\|$.\\
Compute an approximate Hessian $\overline{\nabla^2f}(x_k)$  that satisfies condition \eqref{AccDynH}  with a prefixed probability. Form the model $m_k(x_k,s,\sigma_k)$ defined in \eqref{m}.

 \vspace{2mm}

\item[Step 3: Step calculation. ] Choose $\beta_k\le \beta$. Compute the step $s_k$ satisfying \eqref{mdecr}-\eqref{tc}.\vspace{2mm}

\item[Step 4: Check on  the norm of the trial step. ] If $\| s_k \|< 1$ and ${\rm flag}=1$  and $c>\alpha(1-\beta)\|\overline{\nabla f}(x_k)\|$ 
\begin{itemize}
\item [] set $x_{k+1}=x_k$, $\sigma_{k+1}=\sigma_k$, ${\rm flag}=0$\quad (\textit{unsuccessful iteration}) 
\item[] set $k=k+1$ and go to Step $1$.
\end{itemize}
\vspace{2mm}
 
 \item[Step 5: Acceptance of the trial point and parameters update. ] Compute $f(x_k+s_k)$ and the relative decrease defined in \eqref{rhokdef2}.

If $\rho_k\ge \eta$
\begin{itemize}
\item[] define $x_{k+1}=x_k+s_k$, set $\sigma_{k+1} =
\max[\sigma_{\min},\frac{1}{\gamma} \sigma_k] $.\quad (\textit{successful iteration}) 

 \item[]   If $\|s_k \|\ge 1$ set   ${\rm flag}=1$, otherwise  set ${\rm flag}=0$.


\end{itemize}
else
 \begin{itemize}
 \item[] define $x_{k+1}=x_k$, $\sigma_{k+1}=\gamma\sigma_k. \quad $  (\textit{unsuccessful iteration})
 \end{itemize}
  Set $k=k+1$ and go to Step $1$.
\end{description}
}
\noindent
Some comments on this algorithm are useful at this stage. We first note that the Algorithm \ref{algo} generates a random process
\begin{eqnarray}
\{X_k, S_k, M_k, \Sigma_k,C_k\},\label{sprocess}
\end{eqnarray}
where $C_k=c_k(\omega)$ refers to the random variable for the dynamic Hessian accuracy $c_k$, that  is adaptively  defined in Step 2 of Algorithm \ref{algo}. Since its definition relies on random quantities, $c_k$  constitutes a random variable too.
We recall that, in the deterministic counterpart given in \cite{IMA},  the Hessian approximation $\overline{\nabla^2 f}(x_k)$ computed at iteration $k$ has to satisfy the absolute accuracy requirement 
\eqref{AccDynH}. Here,  this condition is assumed to be satisfied  only with a certain probability (see,  e.g., Assumption \ref{AssAlg}).

\noindent
The main goal is thus to prove that, if $M_k$ is sufficiently accurate with a sufficiently high probability conditioned to the past, then the stochastic process preserves the expected optimal complexity.
 To this scope, the next section is devoted to state the basic probabilistic accuracy assumptions and definitions. In what follows, we use the notation $\mathbb{E}[X]$ to indicate the expected value of a random variable $X$. In addition, given a random event $A$, $Pr(A)$ denotes the probability of $A$, while $\mathbbm{1}_A$ refers to the indicator of the random event $A$ occurring (i.e. $\mathbbm{1}_A(a)=1$ if $a\in A$, otherwise $\mathbbm{1}_A(a)=0$). The notation $A^c$  indicates the complement of the event $A$.



\subsection{Main assumptions on the stochastic ARC algorithm} 
For $k\ge 0$, to formalise the conditioning on the past, let $\mathcal{F}_{k-1}^{M}$ denote the $\hat{\sigma}$-algebra induced by the random variables $M_0$, $M_1$,..., $M_{k-1}$, with $\mathcal{F}_{-1}^{M}=\hat{\sigma}(x_0)$.\\
We first consider the following definitions for measuring the accuracy of the model estimates.\\

\begin{defn}[Accurate model]
\label{AccIk}
A sequence of random models $\{M_k\}$ is said to be $p$-probabilistically sufficiently accurate for Algorithm \ref{algo}, with respect to the corresponding sequence $\{X_k,S_k,\Sigma_k,C_k\}$, if the event 
$I_k=I_k^{(1)}\cap I_k^{(2)}\cap  I_k^{(3)}$, with
\begin{eqnarray}
I_k^{(1)}&=&\left\{\|\overline{\nabla f}(X_k)-\nabla f(X_k)\| \leq \kappa (1-\beta)^2 \left(\frac{\|\overline{\nabla f}(X_k) \|}{\Sigma_k}\right)^2,\quad \kappa>0\right\} , \label{AccG}\\
I_k^{(2)}&=&\left\{\|\overline{\nabla^2 f}(X_k)-\nabla^2 f(X_k)\|\le C_k\right\}, \label{AccH}\\
I_k^{(3)}&=& \left\{\|\overline{\nabla f}(x_k)\|\le \kappa_g, \quad  \|\overline{\nabla^2 f}(x_k)\|\le \kappa_B, \quad  \quad  \kappa_g> 0,~ \kappa_B>0 \right \},  \label{BoundD}
\end{eqnarray}

satisfies 
\begin{eqnarray}
Pr(I_k|\mathcal{F}_{k-1}^{M})=\mathbb{E}[\mathbbm{1}_{I_k}|\mathcal{F}_{k-1}^{M}]\ge p.\label{ProbIk}
\end{eqnarray}
\end{defn}
\noindent
What follows is an assumption regarding the nature of the stochastic information used by Algorithm \ref{algo}.\\

\begin{assumption}\label{AssAlg}
We assume that the sequence of random models  $\{M_k\}$, generated by Algorithm \ref{algo}, is $p$-probabilistically sufficiently accurate for some sufficiently high probability $p\in(0,1]$.
\end{assumption}

\section{Complexity analysis of the algorithm}

For a given level of tolerance $\epsilon$, the aim of this section is to derive a bound on the
expected number of iterations $\mathbb{E}[N_{\epsilon}]$ which is needed, in the worst-case, to reach  an $\epsilon$-approximate first-order stationary point. Specifically,
$N_{\epsilon}$ denotes a random variable corresponding to the number of steps required by the process until  $\|\nabla f(X_k)\|\le \epsilon$ occurs for the first time, namely
\begin{equation}
\label{hittingtime}
N_{\epsilon}=\inf \{k\ge 0~|~ \|\nabla f(X_k)\| \le \epsilon\};
\end{equation}
indeed, $N_{\epsilon}$ can be seen as a stopping time for the stochastic process generated by Algorithm \ref{algo} (see \cite[Definition~2.1]{STR2}).
\noindent
The analysis follows the path of \cite{CartSche17}, but some results need to be proved as for the adopted accuracy requirements for gradient and Hessian and failures in the sense of  Step 4.
It is preliminarly useful to sum up a series of existing lemmas from \cite{CartSche17} and \cite{IMA} and to derive some of their suitable extensions, which will be of paramount importance to perform the complexity analysis of our stochastic method. These lemmas are recalled in the following subsection.

\subsection{Existing and preliminary results}

\noindent
We observe that each iteration $k$ of Algorithm \ref{algo} such that $\mathbbm{1}_{I_k}=1$ corresponds to an iteration of the ARC Algorithm $3.1$ in \cite{IMA}, before termination, except for the fact that in Algorithm \ref{algo} the model \eqref{m} is defined not only using inexact Hessian information, but also considering an approximate gradient. In particular, the nature of the accuracy requirement for the gradient approximation given by \eqref{AccG} is different from the one for the Hessian approximation, namely \eqref{AccH}. In fact, a realisation $c_k$ of the upper bound $C_k$ in \eqref{AccH}, needed to obtain an approximate Hessian $\overline{\nabla^2 f}(x_k)$, is determined by the mechanism of the algorithm and is available when forming the Hessian approximation $\overline{\nabla^2 f}(x_k)$.  On the other hand,  \eqref{AccG} is an implicit condition and can be practically gained computing the gradient approximation within a prescribed absolute accuracy level, that is eventually reduced to recompute the inexact gradient $\overline{\nabla f}(x_k)$; but, in contrast with \cite[Algorithm $4.1$]{CartSche17}, without additional step computation, which is performed only once per iteration at Step $3$ of Algorithm \ref{algo}. We will see that,  for any realisation of the algorithm, if the model is accurate, i.e.   $\mathbbm{1}_{I_k}=1$, then there exist $\delta \ge 0$ and $ \xi_k> 0$ such that
\begin{eqnarray}
\|(\overline{\nabla f}(x_k)-\nabla f(x_k))s_k\| \le \delta \|s_k\|^3, \qquad \|(\overline{\nabla^2 f}(x_k)-\nabla^2 f(x_k))s_k\| \le \xi_k \|s_k\|^2,\nonumber
\end{eqnarray}
which will be  fundamental to recover optimal complexity. At this regard, let us consider the following definitions and state the lemma below.\\

\begin{defn}\label{defset}
With reference to Algorithm \ref{algo}, for all $0\le k\le l$, $l\in\{0,...,N_{\epsilon}-1\}$, we define the events
\vspace{0.1cm}
\end{defn}
\begin{itemize}
\item $\mathcal{S}_k=\{\textrm{iteration~}k~\textrm{is successful}\}$;
\vspace{0.1cm}
\item $\mathcal{U}_{k,1}=\{\textrm{iteration~}k~\textrm{is unsuccessful}~\textrm{in~the~sense~of~Step}~5\}$;
\vspace{0.1cm}
\item $\mathcal{U}_{k,2}=\{\textrm{iteration~}k~\textrm{is unsuccessful}~\textrm{in~the~sense~of~Step}~4\}$.
\end{itemize}
\noindent
We underline that   if $k\in {\cal U}_{k,1}$ then $\rho_k<\eta$, while $k\in {\cal U}_{k,2}$ if and only if  $\| s_k \|< 1$,  ${\rm flag}=1$  and $c>\alpha(1-\beta)\|\overline{\nabla f}(x_k)\|$.
Moreover,  if $\rho_k<\eta$ and a failure in Step 4 does not occur, then $k\in {\cal U}_{k,1}$.

\llem{}{\label{Lemmagk}
Consider any realisation of Algorithm \ref{algo}. Then, at each iteration $k$ such that $\mathbbm{1}_{I_k^{(1)}\cap I_k^{(3)}}=1$ (accurate gradient and bounded inexact derivatives) we have
\begin{equation}
\label{uppboundkgrad}
\|\overline{\nabla f}(x_k)-\nabla f(x_k)\|\le \delta \|s_k\|^2,\qquad \delta \eqdef \kappa \left(\frac{\kappa_B}{\sigma_{\min}}+1\right)\max\left[\frac{\kappa_g}{\sigma_{\min}}, \frac{\kappa_B}{\sigma_{\min}}+1\right],
\end{equation}
and, thus,
\begin{eqnarray}
\|(\overline{\nabla f}(x_k)-\nabla f(x_k))s_k\|\le \delta \|s_k\|^3\label{keygrad}.
\end{eqnarray}
}
\begin{proof} Let us consider $k$ such that $\mathbbm{1}_{I_k^{(1)}\cap I_k^{(3)} }=1$. Using \eqref{tc} we obtain

\begin{eqnarray}
 \beta \|\overline{\nabla f}(x_k)\|&\ge& \|\nabla_s m(x_k, s_k, \sigma_k)\| =\left\| \, \overline{\nabla f}(x_k) + \overline{\nabla^2 f}(x_k) s_k +\sigma_k s_k \|s_k\| \,  \right\| \nonumber \\
 & \ge&  \|\overline{\nabla f}(x_k) \| - \|\overline{\nabla^2 f}(x_k)\|\, \| s_k \| -\sigma_k  \| s_k\|^2. \label{dis}
 \end{eqnarray}
\noindent
We can then distinguish between two different cases. If $\|s_k\| \ge 1$, from \eqref{dis}   and  \eqref{BoundD} we have that
\[
\beta \|\overline{\nabla f}(x_k)\|\ge  \|\overline{\nabla f}(x_k) \| - \|\overline{\nabla^2 f}(x_k)\|\, \| s_k \|^2-\sigma_k  \| s_k\|^2 \ge  \|\overline{\nabla f}(x_k) \| - (\kappa_B+\sigma_k)  \| s_k\|^2
\]
which is equivalent to
\[
\| s_k\|^2 \ge \frac{(1-\beta)\|\overline{\nabla f}(x_k)\|}{\kappa_B+\sigma_k}.
\]
Consequently, by  \eqref{AccG} and    \eqref{BoundD}
\begin{eqnarray}
\|\overline{\nabla f}(x_k)-\nabla f(x_k)\| &\le&  \kappa \left (\frac{1-\beta}{\sigma_k}\right )^2 \|\overline{\nabla f}(x_k)\|^2\le \frac{\kappa \kappa_g (1-\beta)^2  \|\overline{\nabla f}(x_k)\|}{\sigma_k^2\| s_k \|^2}\|s_k  \|^2\nonumber \\
&\le&  \kappa \kappa_g (1-\beta)  \frac{\kappa_B+\sigma_k}{\sigma_k^2}\| s_k  \|^2 \le 
 \kappa \frac{\kappa_g}{\sigma_{\min}}\left(  \frac{\kappa_B}{\sigma_{\min}}+1\right)\| s_k  \|^2 ,\label{uppboundagr0}
 \end{eqnarray}
where in the last inequality we have used that $\beta \in (0,1)$ and $\sigma_k\ge \sigma_{min}$.
If, instead, $\|s_k\| < 1$, inequality \eqref{dis}  and \eqref{BoundD} lead to
\[
\beta \|\overline{\nabla f}(x_k)\|\ge  \|\overline{\nabla f}(x_k) \| - \|\overline{\nabla^2 f}(x_k)\|\, \| s_k \|-\sigma_k  \| s_k\| \ge   \|\overline{\nabla f}(x_k) \|-(\kappa_B+\sigma_k)\|s_k\|,
\]
obtaining that
\begin{equation}
\label{undersk}
\| s_k\| \ge \frac{(1-\beta)\|\overline{\nabla f}(x_k)\|}{\kappa_B+\sigma_k}.
\end{equation}
Hence, by squaring both sides in the above inequality and using \eqref{AccG},  $\beta \in (0,1)$ and $\sigma_k\ge \sigma_{min}$, we obtain
\begin{eqnarray}
\|\overline{\nabla f}(x_k)-\nabla f(x_k)\| &\le&  \kappa \left(\frac{1-\beta}{\sigma_k}\right)^2 \|\overline{\nabla f}(x_k)\|^2= \frac{\kappa (1-\beta)^2  \|\overline{\nabla f}(x_k)\|^2}{\sigma_k^2\| s_k \|^2}\|s_k  \|^2\nonumber \\
&\le&  \kappa\left(\frac{\kappa_B+\sigma_k}{\sigma_k}\right)^2 \|s_k  \|^2 \le  \kappa\left(\frac{\kappa_B}{\sigma_{\min}}+1\right)^2 \|s_k  \|^2 .\label{uppboundagr00}
 \end{eqnarray}
 \noindent
Inequality \eqref{uppboundkgrad} then follows by virtue of  \eqref{uppboundagr0} and \eqref{uppboundagr00}, while \eqref{keygrad} stems from \eqref{uppboundkgrad} by means of the triangle inequality.
\end{proof}

\noindent The following Lemma is a slight modification of \cite[Lemma~3.1]{IMA}.

\llem{}{\label{LemmaCk}
Consider any realisation of Algorithm \ref{algo} and assume that $c\ge\alpha(1-\beta)\kappa_g$. Then, at each iteration $k$ such that  $\mathbbm{1}_{I_k^{(2)}\cap I_k^{(3)}}(1-\mathbbm{1}_{\mathcal{U}_{k,2}})=1$ (successful or unsuccessful in the sense of Step $5$, with accurate Hessian and bounded inexact derivatives) we have
\begin{equation}
\label{uppboundk}
\|\overline{\nabla^2f}(x_k)-\nabla^2 f(x_k)\|\le c_k\le \xi_k \|s_k\|,\qquad \xi_k\eqdef \max[c,\alpha(\kappa_B+\sigma_k)],
\end{equation}
and, thus,
\begin{eqnarray}
\|(\overline{\nabla^2f}(x_k)-\nabla^2 f(x_k))s_k\|\le \xi_k\|s_k\|^2\label{key}.
\end{eqnarray}
}
\begin{proof} Let us consider $k$ such that  $\mathbbm{1}_{I_k^{(2)} \cap I_k^{(3)}}(1-\mathbbm{1}_{\mathcal{U}_{k,2}})=1$. Algorithm \ref{algo} ensures that, if $\|s_k\|\ge 1$, then $c_k=c$ or
\begin{equation}
\label{ckgrad}
c_k=\alpha(1-\beta)\|\overline{\nabla f}(x_k)\|.
\end{equation}
Trivially, \eqref{ckgrad}, $\|s_k\|\ge 1$ and \eqref{BoundD} give
\begin{equation}
\label{uppboundagr1}
\|\overline{\nabla^2f}(x_k)-\nabla^2 f(x_k)\|\le c_k\le \max[c,\alpha(1-\beta)\|\overline{\nabla f}(x_k)\|]\le \max[c,\alpha(1-\beta)\kappa_g]\le c\|s_k\|,
\end{equation}
where we have considered the assumption $c\ge\alpha(1-\beta)\kappa_g$. On the other hand, Step $4$ guarantees the choice 
\begin{equation}
\label{ckgrad_bound}
c_k\le \alpha(1-\beta)\|\overline{\nabla f}(x_k)\|,
\end{equation}
 when $\|s_k\|< 1$. In this case, inequality \eqref{undersk} still holds. Thus, 
\begin{equation}
\label{uppboundagr2}
\|\overline{\nabla^2f}(x_k)-\nabla^2 f(x_k)\| \le  c_k  = \frac{c_k}{\| s_k \|}\|s_k  \| \le  \frac{c_k(\kappa_B+\sigma_k)}{(1-\beta)\| \overline{\nabla f}(x_k)  \|}\| s_k  \|\le \alpha (\kappa_B+\sigma_k)\|s_k\|,
\end{equation}
where  the last inequality is due to \eqref{ckgrad_bound}. Finally, \eqref{uppboundagr1} and \eqref{uppboundagr2} imply \eqref{uppboundk}, while \eqref{key} follows by  \eqref{uppboundk} using the triangle inequality.
\end{proof}

\noindent

\noindent
The next lemma bounds the decrease of the objective function on successful iterations, irrespectively of the satisfaction of the accuracy requirements for gradient and Hessian approximations.

\llem{}{Consider any realisation of Algorithm \ref{algo}. At each iteration $k$ we have
\begin{equation}
\label{Tdecr}
\overline T_2(x_k,0)-\overline T_2(x_k,s_k)> \frac{\sigma_k}{3}\|s_k\|^3\ge  \frac{\sigma_{\min}}{3}\|s_k\|^3>0.
\end{equation}
Hence, on every successful iteration $k$:
\begin{equation}\label{Tdecrsucc}
f(x_k)-f(x_{k+1})>  \eta \frac{\sigma_k}{3} \|s_k\|^3\ge \eta \frac{\sigma_{\min}}{3}\|s_k\|^3>0.
\end{equation}
}
\begin{proof}
We first notice that, by \eqref{mdecr}, we have that $\|s_k\|\neq 0$.  Moreover,   Lemma 2.1 in \cite{Toint1} coupled with (\ref{m})   yields \eqref{Tdecr}.
The second part of the thesis is easily proved taking into account that, if $k$ is successful, then \eqref{Tdecr} implies
\[
f(x_k)-f(x_{k+1})\ge \eta(\overline T_2(x_k,0)-\overline T_2(x_k,s_k))>\eta \frac{\sigma_k}{3}\|s_k\|^3.
\]
\end{proof}

\noindent
As a corollary, because of the fact that $x_{k+1}=x_k$ on each unsuccessful iteration $k$, for any realisation of Algorithm \ref{algo} we have that
\[
f(x_k)-f(x_{k+1})\ge 0.
\]
\noindent
We now show that, if the model is accurate, there exists a constant  $\overline \sigma>0$  such that an iteration is successful or unsuccessful in the sense of Step 4  ($\mathbbm{1}_{I_k}(1-\mathbbm{1}_{\mathcal{U}_{k,1}})=1$), whenever $\sigma_k\ge \overline \sigma$. In other words, it is an iteration at which the regulariser is not increased.

\llem{}{\label{Lemmasigmabar}
Let Assumption \ref{Assf} (ii) hold. Let $\delta$ be given in \eqref{uppboundkgrad}, assume $c\ge\alpha(1-\beta)\kappa_g$ and the validity of \eqref{LipHess}. For any realisation of Algorithm \ref{algo}, if the model is accurate and 
\begin{equation}\label{sigmabar}
\sigma_k\ge \overline{\sigma}\eqdef \max\left[ \frac{6\delta+3\alpha\kappa_B+L_H}{2(1-\eta)-3\alpha},\frac{6\delta+3c+L_H}{2(1-\eta)}\right]>0,
\end{equation}
then the iteration $k$ is successful or a failure in the sense of Step 4 occurs.}
\begin{proof} Let us consider an iteration $k$ such that $\mathbbm{1}_{I_k}(1-\mathbbm{1}_{\mathcal{U}_{k,1}})=1$ and the definition of $\rho_k$ in \eqref{rhokdef2}. Assume that a failure in the sense of Step 4 does not occur. 
 If $\rho_k-1\ge 0$, then iteration $k$ is successful by definition. We can thus focus on the case in which $\rho_k-1< 0$. In this situation, the iteration $k$ is successful provided that $1-\rho_k\le 1-\eta$. From \eqref{LipHess} and the Taylor expansion of $f$ centered at $x_k$ with increment $s$ it first follows that
\begin{equation} \label{Taylorfxksk}
 f(x_k+s) \le f(x_k)+ \nabla f(x_k)^\top s+\frac{1}{2}s^\top\nabla^2 f(x_k)s+\frac{L_H}{6}\|s\|^3.
\end{equation}
Therefore, since $\mathbbm{1}_{I_k}=1$,
\begin{eqnarray}
f(x_k+s_k)-\overline T_2(x_k,s_k)&\le&(\nabla f(x_k)-\overline{\nabla f}(x_k))^\top s_k+\frac12 s_k^\top(\nabla^2 f(x_k)-\overline{\nabla^2 f}(x_k))s_k+\frac{L_H}{6}\|s_k\|^3\nonumber\\
&\le&  \|\overline{\nabla f}(x_k)-\nabla f(x_k)\|\|s_k\|+\frac12\|\overline{\nabla^2 f}(x_k)-\nabla^2 f(x_k)\|\|s_k\|^2+\frac{L_H}{6}\|s_k\|^3\nonumber\\
&\le&  \left(\delta+\frac{L_H}{6}+\frac{\xi_k}{2}\right)\|s_k\|^3,\label{uppf-T}
\end{eqnarray}
where we have used  \eqref{uppboundkgrad} and \eqref{uppboundk}. Thus, by \eqref{uppf-T} and \eqref{Tdecr},
\[
1-\rho_k=\frac{f(x_k+s_k)-\overline T_2(x_k,s_k)}{\overline T_2(x_k,0)-\overline T_2(x_k,s_k)}<\frac{ \left(6\delta+3\xi_k+L_H\right)\|s_k\|^3}{2\sigma_k\|s_k\|^3}=\frac{ 6\delta+3\xi_k+L_H}{2\sigma_k}.
\]
Depending on the maximum in the definition of $\xi_k$ in \eqref{uppboundk}, two different cases can then occur. If $\xi_k=c$, $1-\rho_k \le 1-\eta$, provided that
\[
\sigma_k\ge\frac{6\delta+3c+L_H}{2(1-\eta)}.
\]
Otherwise, if $c<\alpha(\kappa_B+\sigma_k)$, so that $\xi_k=\alpha(\kappa_B+\sigma_k)$, then
\[
1-\rho_k<\frac{ 6\delta+3\alpha(\kappa_B+\sigma_k)+L_H}{2\sigma_k}\le 1-\eta,
\]
provided that
\[
\sigma_k\ge\frac{6\delta+3\alpha \kappa_B+L_H}{2(1-\eta)-3\alpha}.
\]
In conclusion, iteration $k$ is successful if  \eqref{sigmabar} holds. Note that $\overline \sigma$
is a positive lower bound for $\sigma_k$ because of the ranges for the values of $\eta$ and $\alpha$ in \eqref{initialconsts}.
\end{proof}

\noindent
Using some of the results from the proof of the previous lemma, we can now prove the following, giving a crucial relation between the step length $\|s_k\|$ and the true gradient norm $\|\nabla f(x_k+s_k)\|$ at the next iteration.

\llem{}{\label{Lemmapass}Let Assumption \ref{Assf} (ii)-(iii) hold and assume $c\ge\alpha(1-\beta)\kappa_g$. For any realisation of Algorithm \ref{algo}, at  each iteration $k$ such that $\mathbbm{1}_{I_k}(1-\mathbbm{1}_{\mathcal{U}_{k,2}})=1$ (accurate in which the iteration is successful or a failure in the sense of Step 5 occurs), we have
\begin{equation}\label{passtoeps}
\|s_k \|\ge \sqrt{\nu_k\|\nabla f(x_k+s_k)\|},
\end{equation}
for some positive $\nu_k$, whenever $s_k$ satisfies (\ref{tcsub}).   Moreover, \eqref{passtoeps} holds even in case $s_k$ satisfies (\ref{tc.s}) provided that 
 there exists  $L_g>0$ such that 
\begin{equation}\label{Lipgrad}
\|\nabla f(x)-\nabla f(y)\| \le L_g \| x-y\|,
\end{equation}
for all $x$, $y\in\mathbb{R}^n$.
 }
\begin{proof}
Let us consider an iteration $k$ such that $\mathbbm{1}_{I_k}(1-\mathbbm{1}_{\mathcal{U}_{k,2}})=1$. From the Taylor series of $\nabla f(x)$ centered at $x_k$ with increment $s$, 
and the definition of the model \eqref{m}, proceeding as in the proof of Lemma 4.1 in \cite{IMA} we obtain
\begin{eqnarray}
\|\nabla f(x_k+s_k)-\nabla_s \overline T_2(x_k+s_k)\|
&\le& \|\nabla f(x_k)-\overline{\nabla f}(x_k) \| + \| (\nabla^2 f(x_k)-\overline{\nabla^2 f}(x_k))s_k  \|\nonumber \\
&~~~+& \int_0^1 \|\nabla^2 f(x_k+\tau s)-\nabla^2 f(x_k)\| \|s_k\|\,d\tau\nonumber\\
&\le& \left(\delta+\xi_k+\frac{L_H}{2}\right)\|s_k\|^2,\label{normf-T}
\end{eqnarray}
where we have used \eqref{uppboundkgrad}, \eqref{key} and \eqref{LipHess}. Moreover, since $\nabla_s m(x_k, s_k, \sigma_k)=\nabla_s \overline T_2(x_k,s_k)+\sigma_k\| s_k\| s_k$, it follows:
\begin{eqnarray}
\label{normfkk} 
\|\nabla f(x_k+s_k) \| \le \|  \nabla f(x_k+s_k)-\nabla_s \overline T_2(x_k,s_k)\| + \|  \nabla_s m(x_k, s_k, \sigma_k)\| + \sigma_k\| s_k\|^2.
\end{eqnarray} 
As a consequence, the thesis follows from \eqref{normf-T}--\eqref{normfkk} with 
\begin{equation}
\label{zetak1}
\nu_k^{-1}=\left(\delta+\xi_k+\frac{L_H}{2}+\beta+\sigma_k\right)>0,
\end{equation}
when the stopping criterion \eqref{tcsub} is considered.
Assume now that \eqref{tc.s} is used for Step $3$ of Algorithm \ref{algo}. Inequalities \eqref{uppboundkgrad} and \eqref{Lipgrad} imply that
\begin{eqnarray}
\| \overline{\nabla f}(x_k)\| &\le&  \|\overline{\nabla f}(x_k)-\nabla f(x_k) \| + \| \nabla f(x_k) - \nabla f(x_k+s_k) \| + \| \nabla f(x_k+s_k) \| \nonumber\\
&\le& \delta\|s_k\|^2+L_g\|s_k\|+\| \nabla f(x_k+s_k) \|.\label{boundinexgrad}
\end{eqnarray}
By using \eqref{normf-T}--\eqref{normfkk} and 
plugging \eqref{boundinexgrad} into \eqref{tc.s}, we finally have
\[
\| \nabla f(x_k+s_k) \|(1-\beta)\le \left[ (1+\beta)\delta+\xi_k+\frac{L_H}{2}+ \beta L_g+\sigma_k \right]\|s_k\|^2,
\]
which is equivalent to \eqref{passtoeps}, with
\begin{equation}
\label{zetak2}
\nu_k=\frac{1-\beta}{ (1+\beta)\delta+\xi_k+L_H/2+ \beta L_g+\sigma_k }>0.
\end{equation}
\end{proof}

\noindent
It is worth noticing that the global Lipschitz continuity of the gradient, namely, \eqref{Lipgrad}, is  needed only when condition (\ref{tc.s}) is used in Step 3 of Algorithm \ref{algo}.
\noindent
We finally recall a result from \cite{CartSche17} that will be of key importance to carry out the complexity analysis addressed in the following two subsections.

\llem{}{\cite[Lemma~2.1]{CartSche17}
Let $N_{\epsilon}$ be the hitting time defined as in \eqref{hittingtime}. For all $k<N_{\epsilon}$, let $\{I_k\}$ be the sequence of events in Definition \ref{AccIk} so that \eqref{ProbIk} holds. Let $\mathbbm{1}_{W_k}$ be a nonnegative stochastic process such that $\hat{\sigma}(\mathbbm{1}_{W_k})\subseteq \mathcal{F}_{k-1}^M$, for any $k\ge 0$. Then,
\[
\mathbb{E}\left[ \sum_{k=0}^{N_{\epsilon}-1}\mathbbm{1}_{W_k}\mathbbm{1}_{I_k}\right]\ge p \mathbb{E}\left[ \sum_{k=0}^{N_{\epsilon}-1} \mathbbm{1}_{W_k}\right].
\]
Similarly, 
\[
\mathbb{E}\left[ \sum_{k=0}^{N_{\epsilon}-1}\mathbbm{1}_{W_k}(1-\mathbbm{1}_{I_k})\right]\le(1-p)\mathbb{E}\left[ \sum_{k=0}^{N_{\epsilon}-1} \mathbbm{1}_{W_k}\right].
\]
\label{CSlemma21}
}

\subsection{Bound on the expected number of steps with \boldmath$\Sigma_k\ge \overline{\sigma}$} 
In this section we derive an upper bound for the expected number of steps in the process generated by Algorithm \ref{algo} with $\Sigma_k\ge \overline{\sigma}$. Given $l\in\{0,...,N_{\epsilon}-1\}$,
for all $0\le k\le l$,  let us define the event
\[\Lambda_k=\{\textrm{iteration~}k~ \textrm{is such that} ~\Sigma_k< \overline{\sigma}\}\]
and let 
\begin{equation}
\label{defnsigma}
N_{\overline{\sigma}}\eqdef  \sum_{k=0}^{N_{\epsilon}-1}(1-\mathbbm{1}_{\Lambda_k}),\qquad N_{\overline{\sigma}}^{^C}\eqdef  \sum_{k=0}^{N_{\epsilon}-1}\mathbbm{1}_{{\Lambda}_k},
\end{equation}
be the number of steps, in the stochastic process induced by Algorithm \ref{algo}, with $\Sigma_k\ge \overline{\sigma}$ and $\Sigma_k< \overline{\sigma}$, before $N_{\epsilon}$ is met, respectively. In what follows we consider the validity of Assumption \ref{Assf}, Assumption \ref{AssAlg} and the following assumption on $\Sigma_0$.\\

\begin{assumption}\label{Asssigma0} With reference to the stochastic process generated by Algorithm \ref{algo} and the definition of $\overline{\sigma}$ in \eqref{sigmabar}, we assume that
\begin{equation}
\label{Sigma0}
\Sigma_0=\gamma^{-i} \overline{\sigma},
\end{equation}
for some positive integer $i$. We additionally assume that $c\ge\alpha(1-\beta)\kappa_g$.
\end{assumption}

\noindent
By referring to Lemma \ref{CSlemma21} and some additional lemmas from \cite{CartSche17}, we can first obtain an upper bound on $\mathbb{E}[N_{\overline{\sigma}}]$.  In particular, rearranging  \cite[Lemma~2.2]{CartSche17}, given a generic iteration $l$, we derive a bound on the number of iterations successful and unsuccessful  in the sense of Step 4, in terms of the overall number of iterations $l+1$. At this regard, we underline that in case of unsuccessful iterations in Step 4, the value of $\Sigma_k$ is not modified and  such an iteration occurs at most  once between two successful iterations (not necessary consecutive)  with the first one having the norm of the step not smaller than one or once before the first successful iteration of the process (since flag is initially 1).  In fact, a failure in the sense of Step 4 may occur only if 
flag=1; except for the first step, flag is reassigned only at the end of a successful iteration
and can be set to one only in case of successful iteration with  $\|s_k\|\ge 1$ (see Step 5 of Algorithm \ref{algo}), except for the first iteration. If the case 
flag = 1 and $\|s_k\| < 1$ occurs then flag is set to zero, preventing a failure in Step 4 at the subsequent
iteration, and it is not further changed until a subsequent successful iteration.

\llem{}{Assume that   $\Sigma_0< \overline{\sigma}$. Given $l\in\{0,...,N_{\epsilon}-1\}$, for all realisations of Algorithm \ref{algo}, 
\[
\sum_{k=0}^{l}\left(1-\mathbbm{1}_{\Lambda_k}\right) \mathbbm{1}_{\mathcal{S}_k\cup \mathcal{U}_{k,2}}\le \frac23 (l+1).
\]
\label{CSlemma22}
}
\begin{proof}
Each iteration $k$ such that $(1-\mathbbm{1}_{\Lambda_k})\mathbbm{1}_{\mathcal{S}_k\cup \mathcal{U}_{k,2}}=1$ is an iteration with $\Sigma_k\ge\overline{\sigma}$ that can be a successful iteration, leading to $\Sigma_{k+1}=\max[\sigma_{\min},\frac1\gamma \Sigma_k]$ ($\Sigma_k$ is decreased), or an unsuccessful iteration in the sense of Step $4$. In the latter case, $\Sigma_k$ is left unchanged ($\Sigma_{k+1}=\Sigma_k$).
Moreover, $\Sigma_k$  in successful/unsuccessful in the sense of Step $5$  iterations is decreased/increased by the same factor $\gamma$.
 More in depth, since $\Sigma_0< \overline{\sigma}$,  we have two possible scenarios. In the first one we have $\Sigma_k< \overline{\sigma}$, $k=0,\ldots,l$ and the thesis obviously follows. In the second scenario there exists at least one index $k$ such that  $\Sigma_k\ge \overline \sigma$ and at least one unsuccessful iteration at iteration $j\in \{0,\ldots,k-1\}$ in which $\Sigma_k$ has been increased by the factor $\gamma$.  
 In case  $\mathbbm{1}_{ \mathcal{U}_{k,2}}=1$,
 $\Sigma_k$ is left unchanged, flag is set to $0$ and $\mathbbm{1}_{ \mathcal{U}_{k+1,2}}=0$.   Then, at any iteration $j$ such that $\mathbbm{1}_{ \mathcal{U}_{j,1}}=1$
corresponds at most one   successful iteration and one unsuccessful iteration in the sense of Step 4, with $\Sigma_k\ge\overline{\sigma}$ and this yields   the  thesis.%
\end{proof}
\noindent
We note that in the stochastic ARC method in \cite{CartSche17} each iteration can be successful or unsuccessful according to the satisfaction of the decrease condition $\rho_k\ge \eta$. On the contrary, in Algorithm \ref{algo} also failures in Step 4   may occur and this yields the bound $2/3 (l+1)$ in Lemma  \ref{CSlemma22}, while the corresponding bound in  \cite{CartSche17}  is $1/2 (l+1)$.
   
\noindent
As in \cite{CartSche17}, we note that  $\hat \sigma(\mathbbm{1}_{\Lambda_k})\subseteq \mathcal{F}_{k-1}^M$, that is the variable $\Lambda_k$ is fully determined by the first $k-1$ iterations of the Algorithm \ref{algo}. Then,
setting $l=N_{\epsilon}-1$ 
we can rely on Lemma \ref{CSlemma21} (with $W_k=\Lambda_k^c$) to deduce that
\begin{equation}
\label{CSlemma21f}
\mathbb{E}\left[ \sum_{k=0}^{N_{\epsilon}-1} (1-\mathbbm{1}_{\Lambda_k})\mathbbm{1}_{I_k} \right]\ge p \mathbb{E}\left[ \sum_{k=0}^{N_{\epsilon}-1} (1-\mathbbm{1}_{\Lambda_k})\right].
\end{equation}
\noindent
Considering the bound in Lemma \ref{CSlemma22} and the fact that  Lemma \ref{Lemmasigmabar} and the mechanism of Step $4$ in Algorithm \ref{algo} ensure that each iteration $k$ such that 
$\mathbbm{1}_{I_k}=1$ with $\sigma_k\ge \overline{\sigma}$ can be successful or unsuccessful in the sense of Step $4$ (i.e., $\mathbbm{1}_{\mathcal{S}_k\cup \mathcal{U}_{k,2}}=1$), we have that
\[
\sum_{k=0}^{N_{\epsilon}-1} (1-\mathbbm{1}_{\Lambda_k})\mathbbm{1}_{I_k}\le \sum_{k=0}^{N_{\epsilon}-1} (1-\mathbbm{1}_{\Lambda_k})\mathbbm{1}_{\mathcal{S}_k\cup \mathcal{U}_{k,2}}\le \frac23N_{\epsilon}.
\]
Taking expectation in the above inequality and recalling the definition of $N_{\overline{\sigma}}$ in \eqref{defnsigma}, from \eqref{CSlemma21f} we conclude that:
\begin{equation}
\label{bound_Ns}
\mathbb{E}[N_{\overline{\sigma}}] \le \frac{2}{3p}\mathbb{E}[N_{\epsilon}].
\end{equation}
The remaining bound for $ \mathbb{E}\big[N_{\overline{\sigma}}^{^C}\big]$ will be derived in the next section.
\subsection{Bound on the expected number of steps with \boldmath$\Sigma_k< \overline{\sigma}$} 
Let us now obtain an upper bound for $ \mathbb{E}\big[N_{\overline{\sigma}}^{^C}\big]$, with $N_{\overline{\sigma}}^{^C}$ defined in \eqref{defnsigma}. To this purpose, the following additional definitions are needed.
\vspace{0.2cm}
\begin{defn} Let ${\cal U}_{k,1}$,  ${\cal U}_{k,2}$ and ${\cal S}_k$ be as defined in Definition \ref{defset}.
With reference  to the process \eqref{sprocess} generated by Algorithm \ref{algo} let us define:
\vspace{0.2cm}
\begin{itemize}
\item the event $\overline{\Lambda}_k=\{\textrm{iteration~}k~ \textrm{is such that} ~\Sigma_k\le \overline{\sigma}\}$, i.e. $\overline{\Lambda}_k$ is the closure of $\Lambda_k$.
\vspace{0.2cm}
\item $M_1=\sum_{k=0}^{N_{\epsilon}-1}\mathbbm{1}_{\overline{\Lambda}_k}(1-\mathbbm{1}_{I_k})$: number of inaccurate iterations with $\Sigma_k\le \overline{\sigma}$;
\vspace{0.2cm}
\item $M_2=\sum_{k=0}^{N_{\epsilon}-1}\mathbbm{1}_{\overline{\Lambda}_k}\mathbbm{1}_{I_k}$: number of accurate iterations with $\Sigma_k\le \overline{\sigma}$;
\vspace{0.1cm}
\item $N_1=\sum_{k=0}^{N_{\epsilon}-1}\mathbbm{1}_{\overline{\Lambda}_k}\mathbbm{1}_{I_k}\mathbbm{1}_{\mathcal{S}_k}$: number of accurate successful iterations with $\Sigma_k\le \overline{\sigma}$;
\vspace{0.2cm}
\item $N_2=\sum_{k=0}^{N_{\epsilon}-1}\mathbbm{1}_{\overline{\Lambda}_k}\mathbbm{1}_{I_k}\mathbbm{1}_{\mathcal{U}_{k,2}}$: number of accurate unsuccessful iterations, in the sense of $\textrm{Step}~4$, with $\Sigma_k\le \overline{\sigma}$;
\vspace{0.2cm}
\item $N_3=\sum_{k=0}^{N_{\epsilon}-1}\mathbbm{1}_{\Lambda_k}\mathbbm{1}_{I_k}\mathbbm{1}_{\mathcal{U}_{k,1}}$: number of accurate unsuccessful iterations, in the sense of $\textrm{Step}~5$, with $\Sigma_k< \overline{\sigma}$;
\vspace{0.2cm}
\item $M_3=\sum_{k=0}^{N_{\epsilon}-1}\mathbbm{1}_{\overline{\Lambda}_k}(1-\mathbbm{1}_{I_k})\mathbbm{1}_{\mathcal{S}_k}$: number of inaccurate successful iterations, with $\Sigma_k\le \overline{\sigma}$;
\vspace{0.2cm}
\item $S=\sum_{k=0}^{N_{\epsilon}-1}\mathbbm{1}_{\overline{\Lambda}_k}\mathbbm{1}_{\mathcal{S}_k}$: number of successful iterations, with $\Sigma_k\le \overline{\sigma}$;
\vspace{0.2cm}
\item $H=\sum_{k=0}^{N_{\epsilon}-1}\mathbbm{1}_{\mathcal{U}_{k,2}}$: number of unsuccessful iterations in the sense of Step $4$;
\vspace{0.2cm}
\item $U=\sum_{k=0}^{N_{\epsilon}-1}\mathbbm{1}_{\Lambda_k}\mathbbm{1}_{\mathcal{U}_{k,1}}$: number of unsuccessful iterations, in the sense of Step $5$, with $\Sigma_k< \overline{\sigma}$.
\end{itemize}
\label{def2}
\end{defn}
\vspace{0.2cm}
\noindent
It is worth noting that an upper bound on  $\mathbb{E}\big[N_{\overline{\sigma}}^{^C}\big]$ is given, once an upper bound on $\mathbb{E}[M_1]+\mathbb{E}[M_2]$ is provided, since 
\begin{equation}
\label{PlanEN}
 \mathbb{E}\big[N_{\overline{\sigma}}^{^C}\big]\le \mathbbm{E}\left[ \sum_{k=0}^{N_{\epsilon}-1}\mathbbm{1}_{\overline{\Lambda}_k} \right]=\mathbbm{E}\left[ \sum_{k=0}^{N_{\epsilon}-1}\mathbbm{1}_{\overline{\Lambda}_k}(1-\mathbbm{1}_{I_k})+\sum_{k=0}^{N_{\epsilon}-1}\mathbbm{1}_{\overline{\Lambda}_k}\mathbbm{1}_{I_k}\right]=\mathbb{E}[M_1]+\mathbb{E}[M_2].
\end{equation}
where $M_1$ and $M_2$ are given in  Definition \ref{def2}.
Following \cite{CartSche17}, to bound $\mathbb{E}[M_1]$ we can still refer to the central Lemma \ref{CSlemma21} (with $W_k=\overline{\Lambda}_k$), of which the result stated below is a direct consequence.

\llem{}{\cite[Lemma~2.6]{CartSche17} 
With reference to the stochastic process \eqref{sprocess} generated by Algorithm \ref{algo} and the definitions of $M_1$, $M_2$ in Definition \ref{def2},
\begin{equation}
\label{EM1}
\mathbb{E}[M_1]\le \frac{1-p}{p} \mathbb{E}[M_2].
\end{equation}
}

\noindent
Concerning the upper bound for $\mathbb{E}[M_2]$ we observe that 
\begin{equation}
\label{defEM2}
\mathbb{E}[M_2]=\sum_{i=1}^3 \mathbb{E}[N_i]\le  \mathbb{E}[N_1]+ \mathbb{E}[N_2]+ \mathbb{E}[U].
\end{equation}
\noindent
In the following Lemma we provide upper bounds for $N_1$ and $N_2$, given in Definition \ref{def2}.

\llem{}{\label{LemmaN1N2} Let Assumption \ref{Assf} hold and that the stopping criterion  \eqref{tcsub} is used to perform each Step $3$ of Algorithm \ref{algo}. 
With reference to the stochastic process \eqref{sprocess} induced by  the Algorithm
  there exists $\kappa_s>0$ such that
\begin{equation}
\label{boundN1}
N_1\le  \kappa_s(f_0-f_{low})\epsilon^{-3/2}+1.
\end{equation}
Moreover, in case the stopping criterion (\ref{tc.s}) is used in Step 3, 
\eqref{boundN1} still holds provided that 
 there exists  $L_g>0$ such that  \eqref{Lipgrad} is satisfied 
for all $x$, $y\in\mathbb{R}^n$.

Finally, let Assumption \ref{Assf}  (i)-(ii) hold,  independently of the stopping criterion used to perform Step 3 
there exists $\kappa_u>0$ such that
\begin{equation}
\label{boundN2}
 N_2\le  \kappa_u(f_0-f_{low}).
 \end{equation}
 }
\begin{proof}
Taking into account that \eqref{Tdecrsucc} holds for each realisation of Algorithm \ref{algo}, \eqref{passtoeps} is valid for each realisation of Algorithm \ref{algo} with  $\mathbbm{1}_{I_k}(1-\mathbbm{1}_{\mathcal{U}_{k,2}})=1$, recalling that $f(X_k)=f(X_{k+1})$ for all $k\in\mathcal{U}_{k,1}\cup\mathcal{U}_{k,2}$ and setting $f_0\eqdef f(X_0)$, it follows:
\[
\begin{split}
f_0-f_{low}&\ge f_0-f(X_{N_{\epsilon}})=\sum_{k=0}^{N_{\epsilon}-1}(f(X_k)-f(X_{k+1}))\mathbbm{1}_{\mathcal{S}_k}\ge \sum_{k=0}^{N_{\epsilon}-1}\overbrace{\eta \frac{\sigma_{\min}}{3}\|S_k\|^3}^{> 0}\mathbbm{1}_{\mathcal{S}_k}\\
&\ge \sum_{k=0}^{N_{\epsilon}-2}\eta \frac{\sigma_{\min}}{3}\|S_k\|^3\mathbbm{1}_{\mathcal{S}_k}\mathbbm{1}_{I_k} \ge  \sum_{k=0}^{N_{\epsilon}-2} \eta\frac{\sigma_{\min}}{3} \nu_k^{3/2}\| \nabla f(X_{k+1})\|^{3/2} \mathbbm{1}_{\mathcal{S}_k}\mathbbm{1}_{I_k} \\
&\ge  \sum_{k=0}^{N_{\epsilon}-2} \eta\frac{\sigma_{\min}}{3} \nu^{3/2}\| \nabla f(X_{k+1})\|^{3/2} \mathbbm{1}_{\mathcal{S}_k}\mathbbm{1}_{I_k}\mathbbm{1}_{\overline{\Lambda}_k}\\
& \ge (N_1-1) \kappa_s^{-1}\epsilon^{3/2},
\end{split}
\]
in which $\nu_k$ is defined in  \eqref{zetak1}  when $s_k$ satisfies \eqref{tcsub} and in   \eqref{zetak2} when $s_k$ satisfies \eqref{tc.s} and 
\begin{equation}\label{defkappas}
\kappa_s^{-1}\eqdef \eta\frac{\sigma_{\min}}{3} \nu^{3/2}
\end{equation}
where
$$
\nu=\frac{1}{ \delta+\max[c,\alpha(\kappa_B+\overline{\sigma})]+L_H/2+ \beta+\overline{\sigma} }>0,
$$
in case  \eqref{tcsub}  is used and 
 $$
\nu= \frac{1-\beta}{ (1+\beta)\delta+\max[c,\alpha(\kappa_B+\overline{\sigma})]+L_H/2+ \beta L_g+\overline{\sigma} }>0,
$$
whenever \eqref{tc.s} is adopted.
Hence, \eqref{boundN1} holds.

Moreover,  an upper bound for $N_2$ can be obtained taking into account that,
as already noticed, an iteration $k\ge 1$ in the process such that $\mathbbm{1}_{\mathcal{U}_{k,2}}=1$ occurs at most once between two successful iterations with the first one having the norm of the trial step not smaller than $1$, plus at most once  before the first successful iteration in the process (since in Algorithm \ref{algo} flag is initialised at $1$). Therefore, by means of \eqref{Tdecrsucc},
\[
\begin{split}
f_0-f_{low}&\ge f_0-f(X_{N_{\epsilon}})=\sum_{k=0}^{N_{\epsilon}-1}(f(X_k)-f(X_{k+1}))\mathbbm{1}_{\mathcal{S}_k}\ge \sum_{\begin{small}\begin{array}{c} k=0\\ \|S_k\|\ge 1\end{array}\end{small}}^{N_{\epsilon}-1}(f(X_k)-f(X_k+S_k))\mathbbm{1}_{\mathcal{S}_k}\\
&\ge  \eta \frac{\sigma_{\min}}{3}  \sum_{\begin{small}\begin{array}{c} k=0\\ \|S_k\|\ge 1\end{array}\end{small}}^{N_{\epsilon}-1} \mathbbm{1}_{\mathcal{S}_k} \|S_k\|^3\ge \kappa_u^{-1}H,
\end{split}
\]
where $H$ denotes (see Definition \ref{def2}) the number of unsuccessful iterations in the sense of Step $4$. Then, since $H\ge N_2$, \eqref{boundN2} follows.

\end{proof}

\noindent
An upper bound for $U$ can be still derived using \cite[Lemma~2.5]{CartSche17},  provided that \eqref{Sigma0} holds. This is because the process induced by Algorithm \ref{algo} ensures that $\Sigma_k$ is decreased by a factor $\gamma$ on successful steps, increased by the same factor on unsuccessful ones in the sense of Step $5$ and \textit{left unchanged} if an unsuccessful iteration in the sense of Step $4$ occurs.

\llem{}{\cite[Lemma~2.5]{CartSche17}
Consider the validity of \eqref{Sigma0}. For any $l\in\{0,...,N_{\epsilon}-1\}$ and for all realisations of Algorithm \ref{algo}, we have that
\[
\sum_{k=0}^l \mathbbm{1}_{\Lambda_k}\mathbbm{1}_{\mathcal{U}_{k,1}}\le \sum_{k=0}^{l}\mathbbm{1}_{\overline{\Lambda}_k}\mathbbm{1}_{\mathcal{S}_k}+\log_{\gamma}\left( \frac{\overline{\sigma}}{\sigma_0} \right).
\]
}

\noindent
Consequently, considering $l=N_{\epsilon}-1$ and Definition \ref{def2},
\begin{equation}
\label{boundU}
U\le S+ \log_{\gamma}\left( \frac{\overline{\sigma}}{\sigma_0} \right)=N_1+M_3+\log_{\gamma}\left( \frac{\overline{\sigma}}{\sigma_0} \right).
\end{equation}
We underline that the right-hand side in \eqref{boundU} involves $M_3$, that has not been bounded yet. To this aim we can proceed as in \cite{CartSche17}, obtaining that
\begin{equation}
\label{boundEM3}
\mathbb{E}[M_3]\le \frac{1-p}{2p-1}\left(2\mathbb{E}[N_1] + \mathbb{E}[N_2] + \log_{\gamma}\left( \frac{\overline{\sigma}}{\sigma_0} \right)\right).
\end{equation}
In fact, recalling the definition of $M_3$ and \eqref{defEM2}, the inequality \eqref{EM1} implies that
\begin{equation}
\label{1boundEM3}
\mathbb{E}[M_3] \le \mathbb{E}[M_1]\le  \frac{1-p}{p} \mathbb{E}[M_2]  \le \frac{1-p}{p}\left(\mathbb{E}[N_1] + \mathbb{E}[N_2] + \mathbb{E}[U] \right).
\end{equation}
Indeed, taking expectation in \eqref{boundU} and plugging it into \eqref{1boundEM3},
\[
\mathbb{E}[M_3] \le \frac{1-p}{p}\left(2\mathbb{E}[N_1] + \mathbb{E}[N_2] +\mathbb{E}[M_3] + \log_{\gamma}\left( \frac{\overline{\sigma}}{\sigma_0} \right) \right),
\]
which yields\eqref{boundEM3}. The upper bound on $\mathbb{E}[M_2]$ then follows:
\begin{eqnarray}
\mathbb{E}[M_2]&\le&\mathbb{E}[N_1] + \mathbb{E}[N_2] + \mathbb{E}[U]  \le 2\mathbb{E}[N_1] + \mathbb{E}[N_2] +\mathbb{E}[M_3]+ \log_{\gamma}\left( \frac{\overline{\sigma}}{\sigma_0} \right)\nonumber\\
&\le& \left(\frac{1-p}{2p-1}+1\right)\left(2\mathbb{E}[N_1] + \mathbb{E}[N_2] + \log_{\gamma}\left( \frac{\overline{\sigma}}{\sigma_0} \right)\right)\nonumber\\
&=& \frac{p}{2p-1}\left(2\mathbb{E}[N_1] + \mathbb{E}[N_2] + \log_{\gamma}\left( \frac{\overline{\sigma}}{\sigma_0} \right)\right)\nonumber\\
&\le&  \frac{p}{2p-1}\left[ (f_0-f_{low})\left(2\kappa_s\epsilon^{-3/2}+\kappa_u\right)+ \log_{\gamma}\left( \frac{\overline{\sigma}}{\sigma_0} \right)+2 \right],\label{EM2}
\end{eqnarray}
in which we have used \eqref{defEM2}, \eqref{boundN2}, \eqref{boundN1}, \eqref{boundU} and \eqref{boundEM3}. 
Therefore, recalling \eqref{PlanEN} and \eqref{EM1}, we obtain that
\begin{equation}
\label{bound_Nsc}
 \mathbb{E}\big[N_{\overline{\sigma}}^{^C}\big]\le \frac{1}{p}\mathbb{E}[M_2]\le \frac{1}{2p-1}\left[ (f_0-f_{low})\left(2\kappa_s\epsilon^{-3/2}+\kappa_u\right)+ \log_{\gamma}\left( \frac{\overline{\sigma}}{\sigma_0} \right)+2  \right],
\end{equation}
where the last inequality follows from \eqref{EM2}.
We are now in the position to state our final result, providing the complexity of the stochastic method associated with Algorithm \ref{algo}, in accordance with the complexity bounds given by the deterministic analysis of an ARC framework with exact \cite{Toint1} and inexact \cite{ARC2, CGToint, CGTIMA, IMA, BellGuriMoriToin19} function and/or derivatives evaluations.\\

\begin{theorem} {\label{ThSComplexity}Let Assumptions \ref{Assf} and \ref{Asssigma0} hold. Assume that Assumption \ref{AssAlg} holds with $p>2/3$ and that the stopping criterion  \eqref{tcsub} is used to perform each Step $3$ of Algorithm \ref{algo}. Then, the hitting time $N_{\epsilon}$ for the stochastic process generated by Algorithm \ref{algo} satisfies
\begin{equation}\label{complexity_final}
\mathbb{E}[N_{\epsilon}]\le \frac{3p}{(3p-2)(2p-1)}\left[(f_0-f_{low})\left(2\kappa_s\epsilon^{-3/2}+\kappa_u\right)+log_{\gamma}\left(\frac{\overline{\sigma}}{\sigma_0}\right)+2 \right].
\end{equation}
Moreover, in case the stopping criterion (\ref{tc.s}) is used to perform  Step 3, 
\eqref{complexity_final} still holds provided that 
 there exists  $L_g>0$ such that  \eqref{Lipgrad} is satisfied 
for all $x$, $y\in\mathbb{R}^n$.
 }

\end{theorem}
\begin{proof}
By definition (see \eqref{defnsigma}),  $\mathbb{E}[N_{\epsilon}]=  \mathbb{E}\big[N_{\overline{\sigma}}\big]+\mathbb{E}\big[N_{\overline{\sigma}}^{^C}\big]$. Thus, considering \eqref{bound_Ns},
\[
\mathbb{E}[N_{\epsilon}] \le \frac{2}{3p}\mathbb{E}[N_{\epsilon}]+\mathbb{E}\big[N_{\overline{\sigma}}^{^C}\big],
\]
and, hence, by \eqref{bound_Nsc},
\[
\mathbb{E}[N_{\epsilon}] \le \frac{3p}{3p-2}\mathbb{E}\big[N_{\overline{\sigma}}^{^C}\big]= \frac{3p}{(3p-2)(2p-1)}\left[(f_0-f_{low})\left(2\kappa_s\epsilon^{-3/2}+\kappa_u\right)+log_{\gamma}\left(\frac{\overline{\sigma}}{\sigma_0}\right)+2 \right],
\]
which concludes the proof.
\end{proof}

\section{Subsampling scheme for finite-sum minimisation} We now consider the solution of large-scale instances of the finite-sum minimisation problems arising in machine learning and data analysis, modelled by \eqref{finite-sum}. In this context, the approximations  $\overline{\nabla f}(x_k)$ and  $\overline{\nabla^2 f}(x_k)$ to the gradient and the Hessian  used  at Step $1$ and Step $2$ of Algorithm \ref{algo},  respectively, are obtained   by subsampling, using subsets of indexes $\calD_{j,k}$, $j\in\{1,2\}$, randomly and uniformly chosen from $\{1,...,N\}$. I.e.,  for $j\in\{1,2\}$,
 \begin{equation}
 \label{approxBernstein}
  \overline{\nabla^j f}(x_k)
  = \frac{1}{|\calD_{j,k}|} \sum_{i \in \calD_{j,k}} \overline{\nabla^j \varphi_i}(x_k),
  \end{equation}
are used   in place of $ \nabla^j f(x_k)
  = \frac{1}{N} \sum_{i=0}^N \nabla^j \varphi_i(x_k)$.
Specifically, if we want $\overline{\nabla^{j} f}(x_k)$ to be within an accuracy $\tau_{j,k}$ with probability at least $p_j$, $j\in\{1,2\}$, i.e.,
\[
Pr\left(\|\overline{\nabla^{j} f}(x_k)-\nabla^{j} f(x_k)\|\le \tau_{j,k} \right) \ge p_j,
\]
the sample size $|\calD_{j,k}|$ can be determined by using the operator-Berstein inequality introduced in \cite{Tropp}, so that $\overline{\nabla^j f}(x_k)$ takes the form (see \cite{BellGuriMoriToin19}) given by \eqref{approxBernstein}, with \begin{equation}
\label{sizeD}
  |\calD_{j,k}|
  \geq \min\left \{ N,\left\lceil\frac{4\kappa_{\varphi,j}(x_k)}{\tau_{j,k}}
  \left(\frac{2\kappa_{\varphi,j}(x_k)}{\tau_{j,k}}+\frac{1}{3}\right)
  \,\log\left(\frac{d_j}{1-p_j}\right)\right\rceil\right \},
  \end{equation}
  where
  \[
  d_j=\left\{\begin{array}{ll}
         
         n+1, & \tim{if} j=1,\\
         2n,  & \tim{if} j=2,
       \end{array}\right.
  \]
  and under the assumption that, for any $x\in\mathbb{R}^n$, there exist non-negative upper bounds $\{\kappa_{\varphi,j}\}_{j=1}^2$ such that
  \[
   \max_{i \in\ii{N}}\|\nabla^j\varphi_i(x)\| \leq \kappa_{\varphi,j}(x),\qquad j\in\{1,2\}.
  \]
Let us assume that there exist $\kappa_g>0$ and $\kappa_B>0$ such that $\kappa_{\varphi,1}(x)\le \kappa_g$ and  $\kappa_{\varphi,2}(x)\le \kappa_B$ for any $x\in \mathbb{R}^n$.
  Since the subsampling procedures used at iteration $k$ to get $\calD_{1,k}$ and $\calD_{2,k}$ are independent, it follows that when $\{\tau_{j,k}\}_{j=1}^2$ are chosen as the right-hand sides in \eqref{AccG} and \eqref{AccH}, respectively, the builded model \eqref{m} is $p$-probabilistically $\delta$-sufficiently accurate with $p=p_1p_2$. Therefore, a practical version of Algorithm \ref{algo} is for instance given by adding a suitable termination criterion and modifying the first three steps of Algorithm \ref{algo} as reported in Algorithm 4.1 below. 
  
  \algo{algodet}{Modified Steps \boldmath$0-2$ of Algorithm \ref{algo}}
{\vspace*{-0.3 cm}
\begin{description}
\item[Step 0: Initialisation.]
  An initial point $x_0\in\mathbb{R}^n$  and an initial regularisation parameter $\sigma_0>0$
  are given, as well as an accuracy level  $\epsilon \in (0,1)$.  The constants $\beta$, $\alpha$,  $\eta$, $\gamma$, $\sigma_{\textrm{min}}$, $\kappa$, $\tau_0$, $\kappa_{\tau}$ and $c$ are also given such that
\begin{eqnarray}
 &0<\beta, \kappa_\tau<1,  \quad 0\le \alpha< \frac 2 3, \quad\sigma_{\min}\in (0, \sigma_0], \nonumber \\
 &0<\eta < \frac{2-3\alpha}{2},\quad \gamma>1,\quad \kappa\in[0,1), \quad \tau_0>0,\quad c>0.\nonumber
\end{eqnarray}
  Compute $f(x_0)$ and set $k=0$, ${\rm flag}=1$.
\vspace{2mm}
\item[Step 1: Gradient approximation. ] Set $i=0$ and initialise $\tau_{1,k}^{(i)}=\tau_0$. Do
\begin{itemize}
\item[1.1] compute $\overline{\nabla f}(x_k)$ such that \eqref{approxBernstein}--\eqref{sizeD} are satisfied with $j=1$, $\tau_{1,k}=\tau_{1,k}^{(i)}$;
\item[1.2] if $\tau_{1,k}^{(i)}$ $\le \kappa (1-\beta)^2 \left(\frac{\|\overline{\nabla f}(x_k)\|}{\sigma_k}\right)^2$, go to Step $2$;
\item[] else, set $\tau_{1,k}^{(i+1)}=\kappa_{\tau}\tau_{1,k}^{(i)}$, increment $i$ by one and go to Step $1.1$;
\end{itemize}

\vspace{2mm}

\item[Step 2: Hessian approximation (model costruction). ] If ${\rm flag}=1$ set  $c_{k}=c$, else set $c_{k}=\alpha(1-\beta)\|\overline{\nabla f}(x_{k})\|$.\\
Compute $\overline{\nabla^2f}(x_k)$ using \eqref{approxBernstein}--\eqref{sizeD} with $j=2$, $\tau_{2,k}=c_k$ and form the model $m_k(s)$ defined in \eqref{m}.\vspace{2mm}

\end{description}
}
\noindent
Concerning the gradient estimate, the scheme computes (Step $1$) an approximation $\overline{\nabla f}(x_k)$ satisfying the accuracy criterion 

\begin{equation}
\label{err_relG}
 \|\overline{\nabla f}(x_k)-\nabla f(x_k)\|\le \kappa(1-\beta)^2\left(\frac{\|\overline{\nabla f}(x_k)\|}{\sigma_k}\right)^2,
 \end{equation}
 
which is independent of the step computation and based on the knowable quantities $\kappa$, $\beta$ and $\sigma_k$. This is done by reducing the accuracy $\tau_{1,k}^{(i)}$ and repeating the inner loop at Step $1$, until the fulfillment of the inequality at Step $1.2$. We underline that condition \eqref{err_relG} is guaranteed by the algorithm, since \eqref{sizeD} is a continuous and increasing function with respect to $\tau_{j,k}$, for fixed $j=1$, $k$, $p_j$ and $N$; hence,  there exists a sufficiently small $\overline{\tau}_{1,k}$ such that the right-hand side term in \eqref{sizeD} will reach, in the worst-case, the full sample size $N$, yielding $\overline{\nabla f}(x_k)=\nabla f(x_k)$. 
Moreover, if the stopping criterion $\|\overline{\nabla f}(x_k)\|\le \epsilon$ is used,   the loop is ensured to terminate also whenever the predicted accuracy requirement $\tau_{1,k}^{(i)}$ becomes smaller than 
$\kappa (\frac{1-\beta}{\sigma_k})^2\epsilon^2$. 
 On the other hand, in practice, we expect to use a small number of samples in the early stage of the iterative process, when the norm of the approximated gradient is not yet small. To summarise, if without loss of generality we assume that $\overline{\tau}_{1,k}\ge \hat{\tau}$ at each iteration $k$, we conclude that, in the worst case, Step $1$ will lead to at most $\lfloor \log(\hat{\tau})/\log(\kappa_{\tau}\tau_0)\rfloor+1$ computations of $\overline{\nabla f}(x_k)$. The Hessian approximation $\overline{\nabla^2f}(x_k)$ is, instead, defined at Step $2$ and its computation relies on the reliable value of $c_k$.  We remark that at iteration $k$ we have that:
\begin{itemize}
\vspace{0.1cm}
\item $\overline{\nabla^2 f}(x_k)$ is computed only once, irrespectively of the approximate gradient computation considered at Step $1$;\vspace{0.1cm}
\item a finite loop is considered at Step $1$ to obtain a gradient approximation satisfying \eqref{err_relG}, where the right-hand side is independent of the step length $\|s_k\|$, thou implying \eqref{uppboundkgrad}--\eqref{keygrad}. Hence, the gradient approximation is fully determined at the end of Step $1$ and further recomputations due to the step calculation (see Algorithm \ref{algo}, Step $3$) are not  required. 
\end{itemize}
We conclude this section by noticing that each iteration $k$ of Algorithm \ref{algo} with the modified steps introduced in Algorithm \ref{algodet} can indeed be seen as an iteration of Algorithm \ref{algo} where the sequence of random models $\{M_k\}$ is $p$-probabilistically sufficiently accurate in the sense of Definition \ref{AccIk}, with $p=p_1p_2$, and an iteration of \cite[Algorithm~3.1]{IMA}, when $\kappa=0$ is considered in \eqref{AccG} (exact gradient evaluations).

\section{Numerical tests}
In this section we  analyse the behaviour of the Stochastic ARC Algorithm  (Algorithm  \eqref{algo}). Inexact gradient and Hessian evaluations are performed as  sketched in 
modified Steps 0-2 of  Algorithm  \ref{algodet}. The performance of the proposed algorithm is compared with that of  the corresponding
version in  \cite{IMA}  employing exact gradient, with the aim to provide  numerical evidence that adding a further source of inexactness in gradient computation is beneficial in terms of computational cost saving. 
We consider nonconvex finite-sum minimisation problems. This is, in fact, a highly frequent scenario when dealing with
binary classification tasks arising in machine learning applications.  More in depth, given a training set of $N$ features  $a_i\in\mathbb{R}^n$  and corresponding labels $y_i$,  $i=1,\ldots,N$, 
we solve the following minimisation problem:
\begin{equation}
\label{minloss}
\min_{x\in\mathbb{R}^n} f(x)= \min_{x\in\mathbb{R}^n}  \frac{1}{N}\sum_{i=1}^N{\varphi_i(x)}= \min_{x\in\mathbb{R}^n}\frac 1 N \sum_{i=1}^N \left( y_i-\sigma\left(a_i^\top x\right) \right)^2,
\end{equation}
where 
\begin{equation}
\label{sigm}
\sigma(a^\top w)=\frac{1}{1+e^{-a^\top w}},\qquad a,w\in\mathbb{R}^n.
\end{equation}
That is we use  the sigmoid function \eqref{sigm}
as the model for predicting the values of the labels and the least-squares loss as a measure of the error on such predictions, that has to be minimised by approximately solving \eqref{minloss} in order to come out with the parameter vector $x$, to be used for label predictions on new untested data. Moreover, a  number $N_T$ of testing data $\{\overline a_i,\overline y_i\}_{i=1}^{N_T}$ is used to validate the computed model. The values $\sigma(\overline{a}_i^\top x)$ are used to predict the testing labels $\overline y_i$, $i\in\{1,...,N_T\}$, and the corresponding error, measured by
$ 
\frac{1}{N_T} \sum_{i=1}^{N_T} \left(\overline y_i-\sigma\left(\overline a_i^T x\right) \right)^2,
$ 
is computed.
Implementation issues concerning the considered procedures are the object of Subsection \ref{Impl_iss}, while statistics of our runs are discussed in Subsection \ref{Num_res}.


\subsection{Implementation issues}
\label{Impl_iss}

The implementation of the main phases of Algorithm \eqref{algo}, equipped with the modified steps in Algorithm \eqref{algodet}, respects the following specifications.\\
According to \cite[Algorithm 3.1]{IMA}, the cubic regularisation parameter is initially $\sigma_0=10^{-1}$, its minimum value is $\sigma_{\min}=10^{-5}$ and the initial guess vector $x_0=(0,...,0)^\top \in\mathbb{R}^n$ is considered for all runs. 
Moreover, the probability of success $p_j$ in \eqref{sizeD} is set equal to $0.8$, for $j\in\{1,2\}$, while the parameters $\alpha$, $\beta$, $\epsilon$, $\eta$ and $\gamma$ are fixed as $\alpha=0.1$, $\beta=0.5$, $\epsilon=5\cdot 10^{-3}$, $\eta=0.8$ and $\gamma=2$. The latter two correspond to the values of $\eta_2$ and $\gamma_3$ considered in \cite[Algorithm 3.1]{IMA}, respectively. The minimisation of the cubic model at Step $3$ of Algorithm \ref{algo} is performed by the Barzilai-Borwein gradient method \cite{bb} combined with a nonmonotone linesearch following the proposal in \cite{blms}. The major per iteration cost of such Barzilai-Borwein process is one Hessian-vector product, needed to compute the gradient of the cubic model. The threshold used in the termination criterion (\ref{tc}) is $\beta_k=0.5$, $k\ge 0$. As for \cite[Algorithm 3.1]{IMA}, we impose a maximum of $500$ iterations and a successful termination is declared when the following condition is met:
\[
\|\overline{\nabla f}(x_k)\|\le \epsilon, \quad k\ge 0.
\]

\noindent
In case $\|\overline{\nabla f}(x_k)\|\le \epsilon$ and the model is accurate, then  by  \eqref{AccG}
$$
\|{\nabla f}(x_k)\|\le \|\overline{\nabla f}(x_k)\|+\|{\nabla f}(x_k)-\overline{\nabla f}(x_k)\|\le \overline \epsilon := \epsilon+\kappa[(1-\beta)/\sigma_{min}]^2  \epsilon^2
$$
and, hence, $x_k$ is an $\overline \epsilon$-approximate first-order optimality point. Since the model is accurate with probability at least $p$,   $x_k$ is  an $\overline \epsilon$-approximate first-order optimality point with probability at least $p$.
We further  note that the exact gradient and the Hessian of the component functions $\varphi_i(x)$, $i\in\{1,...,N\}$, are given by: 
\begin{eqnarray}
&&\nabla \varphi_i(x)=-2 e^{-a_i^\top x}\left(1+e^{-a_i^\top x}\right)^{-2}\left(y_i-\left(1+e^{-a_i^\top x}\right)^{-1}\right)a_i,\label{der1phi}\\
& & \nabla^2 \varphi_i(x)=-2 e^{-a_i^\top x}\left(1+e^{-a_i^\top x}\right)^{-4}\left(y_i\left(\left(e^{-a_i^\top x}\right)^2-1\right)+1-2e^{-a_i^\top x}\right)a_ia_i^\top.\label{der2phi}
\end{eqnarray}
\noindent
Then, the gradient and the Hessian approximations $\overline{\nabla^j f}(x_k)$, $j\in\{1,2\}$, computed at Step $1$ and Step $2$ of Algorithm \ref{algodet} according to \eqref{approxBernstein}--\eqref{sizeD}, involve the constants
\begin{eqnarray}
\nonumber \kappa_{\varphi,1}(x_k)&=&\max_{i\in\{1,...,N\}}\left\{2 e^{-a_i^\top x_k}\left(1+e^{-a_i^\top x_k}\right)^{-2}\left|y_i-\left(1+e^{-a_i^\top x_k}\right)^{-1}\right| \|a_i\|\right\},\nonumber\\
\kappa_{\varphi,2}(x_k)&=&\max_{i\in\{1,...,N\}}\left\{2e^{-a_i^\top x_{k}}\left(1+e^{-a_i^\top x_{k}}\right)^{-4}\left|y_i\left(\left(e^{-a_i^\top x_{k}}\right)^2-1\right)+1-2e^{-a_i^\top x_{k}}\right| \|a_i\|^2\right\},\nonumber
\end{eqnarray}
whose computations can indeed be an issue in theirselves. Nevertheless, thank to the exactness and the specific form (see \eqref{minloss}) of the function evaluation $f(x_k)$, the values $a_i^\top x_k$, $1\le i\le N$, are available at iteration $k$ and, hence, $\kappa_{\varphi,j}(x_k)$, $j\in\{1,2\}$, can be determined at the (offline) extra cost of computing $\|a_i\|^j$, $j\in\{1,2\}$, for $1\le i\le N$. As in \cite[Subsection 8.2]{IMA}, the value of $c$ used in \eqref{ck}, in order to reduce the iteration computational cost whenever $\|s_k\|\ge 1$,  is such that  $|\mathcal{D}_{2,0}|$ computed via \eqref{sizeD} for $j=2$, with $\tau_{2,0}=c$ (first approximation of the Hessian), satisfies $|\mathcal{D}_{2,0}|/N=0.1$. We indeed start using the $10\%$ of the examples to approximate the Hessian. Concerning the gradient approximation performed at Step $1$ of Algorithm \ref{algodet},  the value of $\tau_0$ is chosen in order to use  a prescribed percentage of the number of training samples $N$ to obtain  $\overline{\nabla f}(x_0)$. In all runs, such a percentage has been set to $0.4$. Then, we proceeded as follows. We computed $\overline{\nabla f}(x_0)$, via \eqref{approxBernstein}, with $j=1$ and $|\mathcal{D}_{1,0}|/N=0.4$. Then, we compute $\tau_0$ so that  \eqref{sizeD}, with $\tau_{1,0}=\tau_0$ is satisfied as an equality. Finally, 
 the value of $\kappa$ at Step $1.2$ of Algorithm \ref{algodet} has been correspondingly set to $4\tau_{1,0}^{(0)}\left(\sigma_0/\|\overline{\nabla f}(x_0)\|\right)^2$, with $\tau_{1,0}^{(0)}=\tau_0$. This way, the acceptance criterion of Step $1.2$ is satisfied without further inner iterations (i.e., for $i=0$), when $k=0$, and $\tau_0$ is indeed considered as the starting accuracy level for gradient approximation at each execution of Step $1$ of Algorithm \ref{algodet}. We will hereafter refer to such implementation of Algorithm \eqref{algo} coupled with Algorithm \ref{algodet} as $SARC$. The numerical tests of this section compare $SARC$ with the corresponding variant in \cite[Algorithm $3.1$]{IMA}, namely \textit{ARC-Dynamic}, employing exact gradient evaluations, with $\gamma_1=1/\gamma$, $\gamma_2=\gamma_3=\gamma$ and $\eta_1=\eta_2=\eta$. It is worth noticing that the problem \eqref{minloss} arises in the training of an artificial neural network with no hidden layers and zero bias. Nevertheless, to cover the general situation where $SARC$ algorithm is applied to more complex neural networks, we have followed the approach in \cite{Roosta_2p} for what concerns the cost measure. Going into more details, at the generic iteration $k$, we count  the $N$  forward propagations needed to evaluate the objective in \eqref{minloss} at $x_k$ has a unit Cost Measure (CM), while the evaluation of the  approximated gradient at the same point requires $|\mathcal{D}_{1,k}|$ additional backward propagations at the weighed cost $|\mathcal{D}_{1,k}|/N$ CM.  Moreover, each vector-product  $\overline{\nabla^2 f}(x_k)v$ ($v\in\mathbb{R}^n$), needed at each iteration of the Barzilai-Borwein method used 
to minimise the cubic model at Step $3$ of Algorithm \ref{algo},   is performed via finite-differences, leading to  additional  $|\mathcal{D}_{2,k}|$ forward and backward propagations to compute $\overline{\nabla f} (x_k+hv)$, ($h\in\mathbb{R}^+$), at the price of the weighted cost $2|\mathcal{D}_{2,k}|/N$ CM and a potential extra-cost  $|\mathcal{D}_{2,k}\smallsetminus(\mathcal{D}_{1,k}\cap \mathcal{D}_{2,k})|/N$  CM to approximate ${\nabla f}(x_k)$ via uniform subsampling using the samples in $\mathcal{D}_{2,k}$.  This latter approximation is computed once at the beginning of the Barzilai-Borwein procedure. Therefore, denoting by $r$ the number of Barzilai-Borwein iterations at iteration $k$, the increase of the CM at the $k$-th iteration of \textit{ARC-Dynamic} and \textit{SARC}  related to the derivatives computation is reported in Table \ref{cost}.

\begin{small}
\begin{table}[h]
\begin{center}
\begin{tabular}{cccccc}
\toprule
\textit{ARC-Dynamic} & \textit{SARC} \\  
\midrule
$1+2|\mathcal{D}_{2,k}|r/N$ &   $\left(|\mathcal{D}_{1,k}|+2|\mathcal{D}_{2,k}|r+|\mathcal{D}_{2,k}\smallsetminus(\mathcal{D}_{1,k}\cap \mathcal{D}_{2,k})|\right)/N$ \\
\bottomrule
\end{tabular}
\caption{Increase of the CM at the $k$-th iteration of \textit{ARC-Dynamic} and \textit{SARC}  related to the derivatives computation; $r$ denotes the number of performed Barzilai-Borwein iterations.}\label{cost} 

\end{center}
\end{table}
\end{small}

\noindent
We will refer to the Cost Measure at Termination (CMT) as the main parameter to evaluate the efficiency of the method within the numerical tests of the next section. The algorithms have been implemented in Fortran language and run on an Intel Core i5, $1.8$ GHz $\times~1$ CPU, $ 8$ GB RAM.

\subsection{Numerical results}
\label{Num_res}
In this section we finally report  statistics of the numerical tests performed by \textit{SARC} and \textit{ARC-Dynamic} on the set of synthetic datasets from \cite{IMA,bollapragada}, whose   main characteristics are recalled in Table \ref{TableSynth}. They provide moderately ill-conditioned problems (see, e.g., Table \ref{TableSynth}) and motivate the use of second order methods. 

\begin{small}
\begin{table}[h] 
\begin{center}
\begin{tabular}{ccccc}
\toprule
Dataset & Training~$N$ & $n$  & Testing $N_T$ & $cond$   \\  
\midrule
Synthetic1 &   9000 & 100 & 1000 &  $2.5\cdot10^4$  \\
Synthetic2 &   9000 & 100 & 1000 &  $1.4\cdot10^5$  \\
Synthetic3 &   9000 & 100 & 1000 &  $4.2\cdot10^7$  \\
Synthetic4 &   90000 & 100 & 10000 &  $4.1\cdot10^4$ \\
{Synthetic6} &   {90000} & {100} & {10000} &  {$5.0\cdot10^6$} \\
\bottomrule
\end{tabular}
\caption{Number  of training samples ($N$),  feature dimension  ($n$),  number of testing samples ($N_T$),
$2$-norm condition number  of the Hessian matrix at computed solution ($cond$).}\label{TableSynth}
\end{center}
\end{table}
\end{small}
\noindent
For fair comparisons, the values of $c$ used for each dataset in Table \ref{TableSynth} to build the Hessian approximation according to Step $2$ of Algorithm \ref{algodet} are chosen as in \cite[Table $8.1$]{IMA}.\\
\noindent
In Table  \ref{TableSynthALL} we report,  for both  $SARC$ and \textit{ARC-Dynamic} algorithms, the total number of iterations ({\rm n-iter}), the value of Cost Measure at Termination ({\rm CMT}) and the mean percentage of saving ({\rm Save-M}) obtained by  \textit{SARC}  with respect to \textit{ARC-Dynamic} on the synthetic datasets listed in Table \ref{TableSynth}. Since the selection of the subsets $\mathcal{D}_{j,k}$, $j\in\{1,2\}$, in \eqref{sizeD} is uniformly and randomly made at each iteration of the method, statistics in the forthcoming tables are averaged over $20$ runs.

\begin{small}
\begin{table}[h]   
\begin{center}
\begin{tabular}{lcc|ccc}
\toprule
Dataset & \multicolumn{2}{c}{\textit{ARC-Dynamic}} & \multicolumn{3}{c}{\textit{SARC}} \\  
&  n-iter  & CMT & n-ter & CMT   & Save-M\\
\midrule
Synthetic1 &   11.1  & 130.84 & 10.0 & ~95.27 & ~27\% \\ 
Synthetic2  &  10.6  & 109.56 & 10.2 & ~93.08  & ~15\% \\ 
Synthetic3 &   11.2  & 109.64 & 10.0 & ~97.52 & ~11\% \\ 
Synthetic4 &   11.0  & 124.07 & 10.4 & 100.48 &  ~19\% \\
{Synthetic6} &   {10.0} &   {84.18} &  {10.1} & {106.31} & $-26\%$\\
\bottomrule
\end{tabular}
\caption{Synthetic datasets. The columns are divided in two different groups. \textit{ARC-Dynamic}: average number of iterations ({\rm n-iter}) and CMT. \textit{SARC}: average number of iterations ({\rm n-iter}), CMT and mean percentage of saving ({\rm Save-M}) obtained by \textit{SARC} over \textit{ARC-Dynamic}. Mean values over $20$ runs. }   \label{TableSynthALL}
\end{center}
\end{table}
\end{small} 

\begin{small}
\begin{table}[h]   
\begin{center}
\begin{tabular}{lccccc}
\toprule
Method & Synthetic1  & Synthetic2 & Synthetic3 & Synthetic4 & {Synthetic6}\\
\midrule
\textit{ARC-Dynamic}            & 94.34\%  & 92.68\% & 94.64\%  & 95.52\% & {$93.82\%$} \\
\textit{SARC}   & 93.18\%  & 92.44\% & 93.62\%  & 94.61\% &  {$93.70\%$} \\
\bottomrule
\end{tabular}
\caption{Synthetic datasets. Binary classification rate at termination on the testing set employed by \textit{ARC-Dynamic} and \textit{SARC}, mean values over $20$ runs.}   \label{BinAsyntheticdataset}
\end{center}
\end{table}
\end{small}

\noindent
Table \ref{TableSynthALL} shows that  the novel adaptive strategy employed by $SARC$ results more efficient than \textit{ARC-Dynamic}, reaching an $\epsilon$-approximate first-order stationary point at a lower CMT, 
in all cases  except from Synthetic6.  This is obtained without affecting the classification accuracy  on the testing sets as it is shown in Table \ref{BinAsyntheticdataset}, where the average binary accuracy on the testing sets achieved by methods under comparison is reported. \noindent

\noindent
To give more evidence of the gain in terms of CMT provided by  $SARC$ on Synthetic1-Synthetic4 along the iterative process,  we display in Figure \ref{Perf1KL} the decrease of the training and the testing loss versus the adopted cost measure CM, while Figure \ref{GDsyntehticdata} is reserved to the plot of the gradient norm versus CM. For such figures, a representative plot is considered among each series of $20$ runs obtained by $SARC$ and \textit{ARC-Dynamic} on each of the considered dataset.

\begin{figure}[h]
\centering
\includegraphics[width=%
0.49\textwidth]{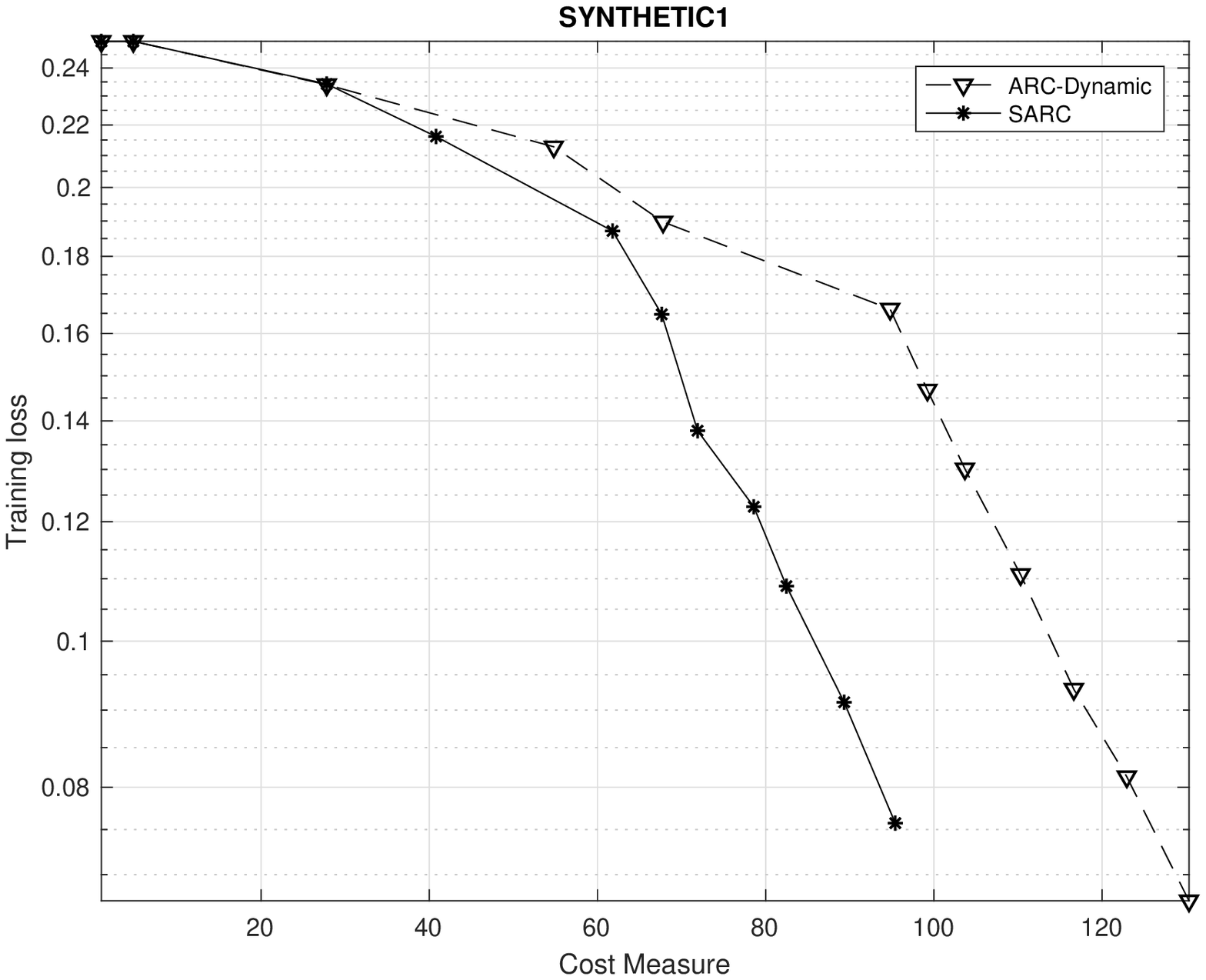}
\includegraphics[width=%
0.49\textwidth]{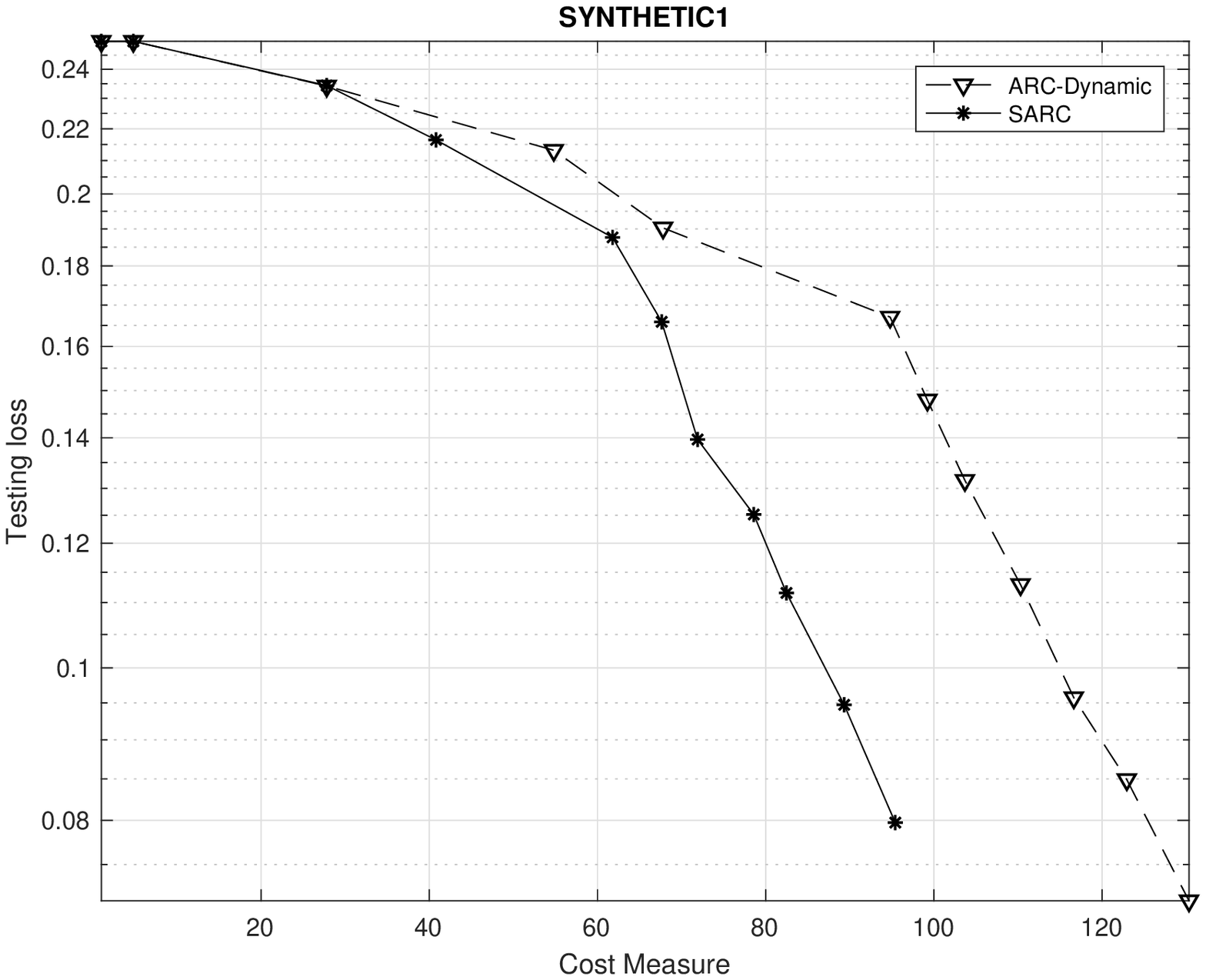}\
\includegraphics[width=%
0.49\textwidth]{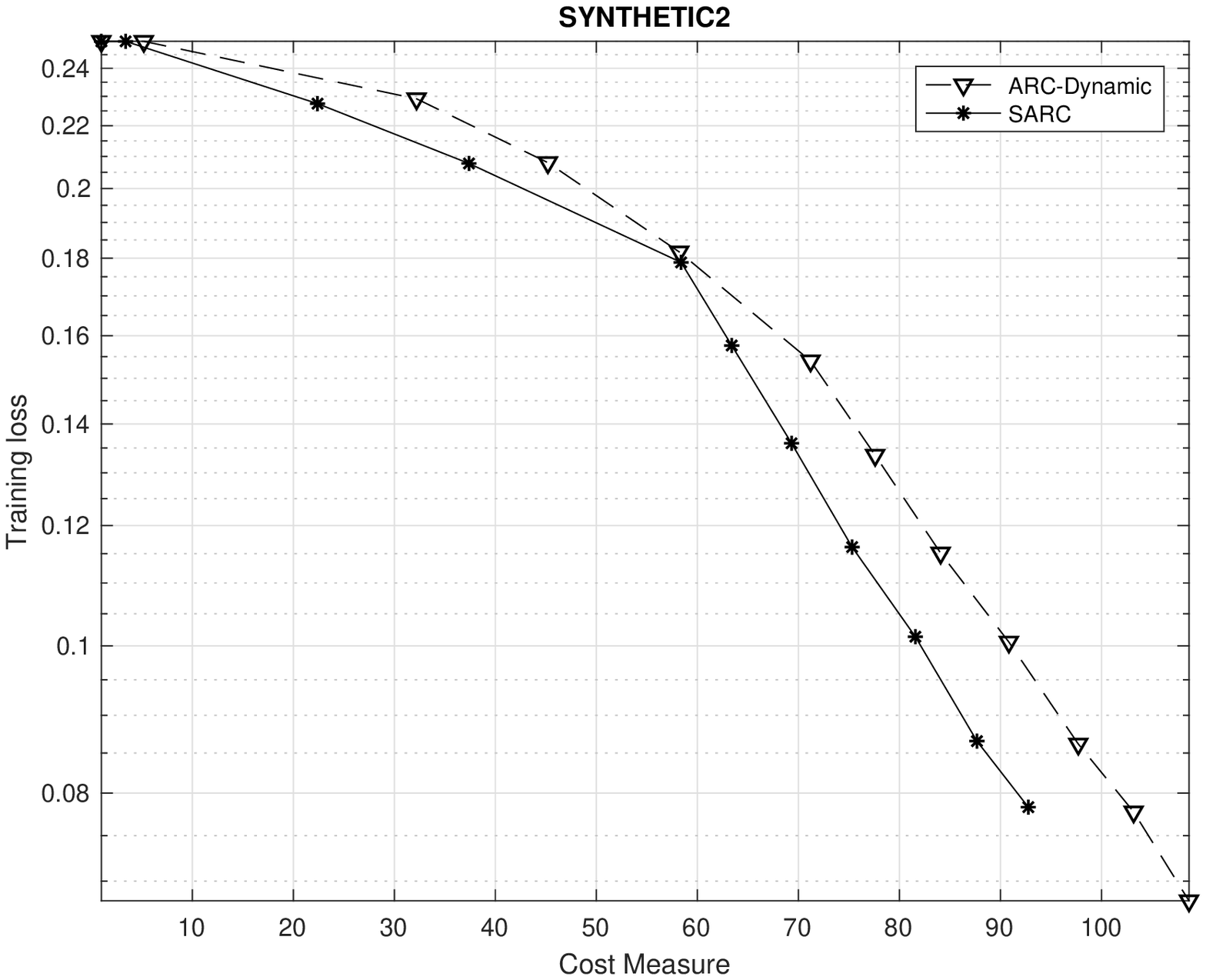}
\includegraphics[width=%
0.49\textwidth]{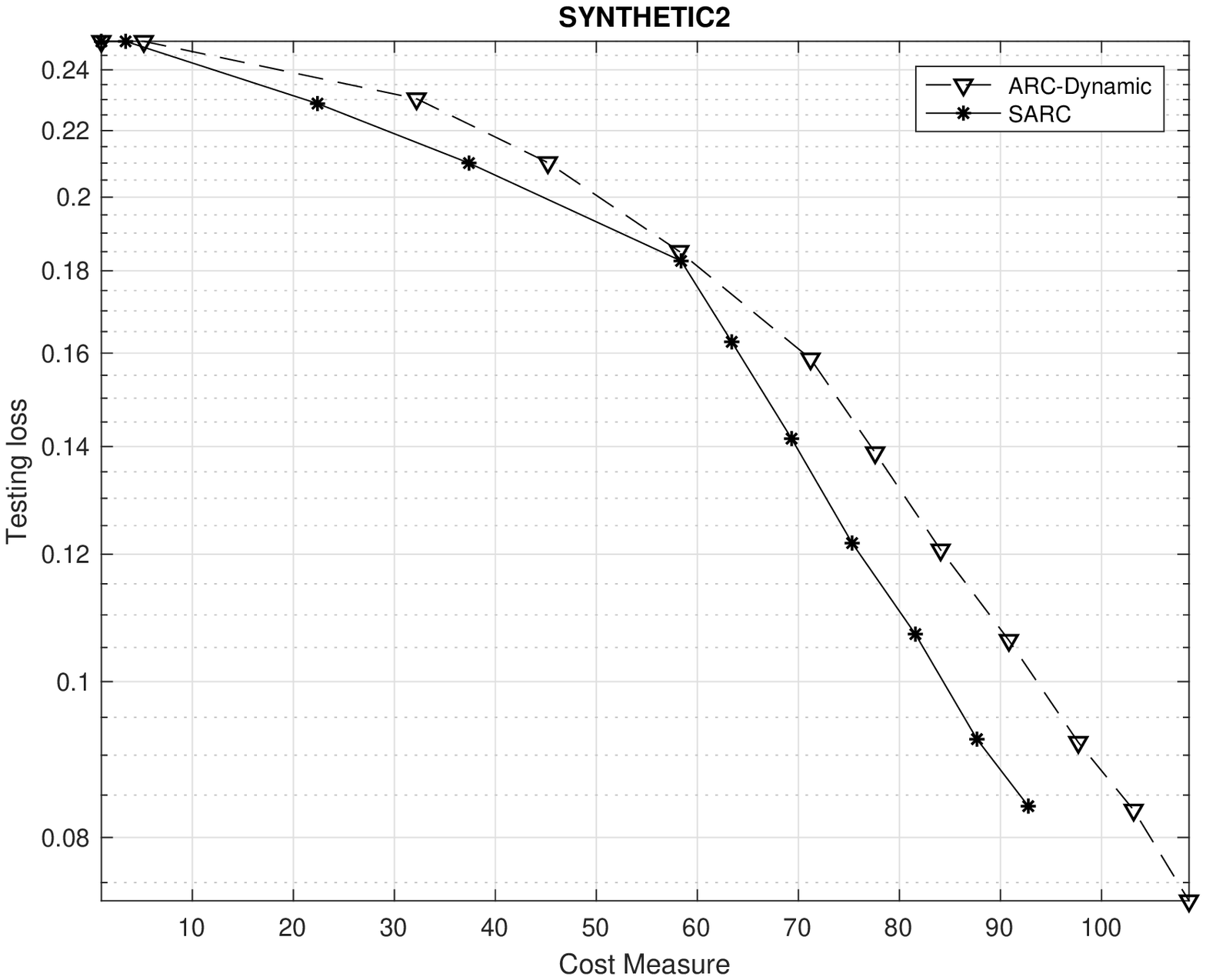}
\includegraphics[width=%
0.49\textwidth]{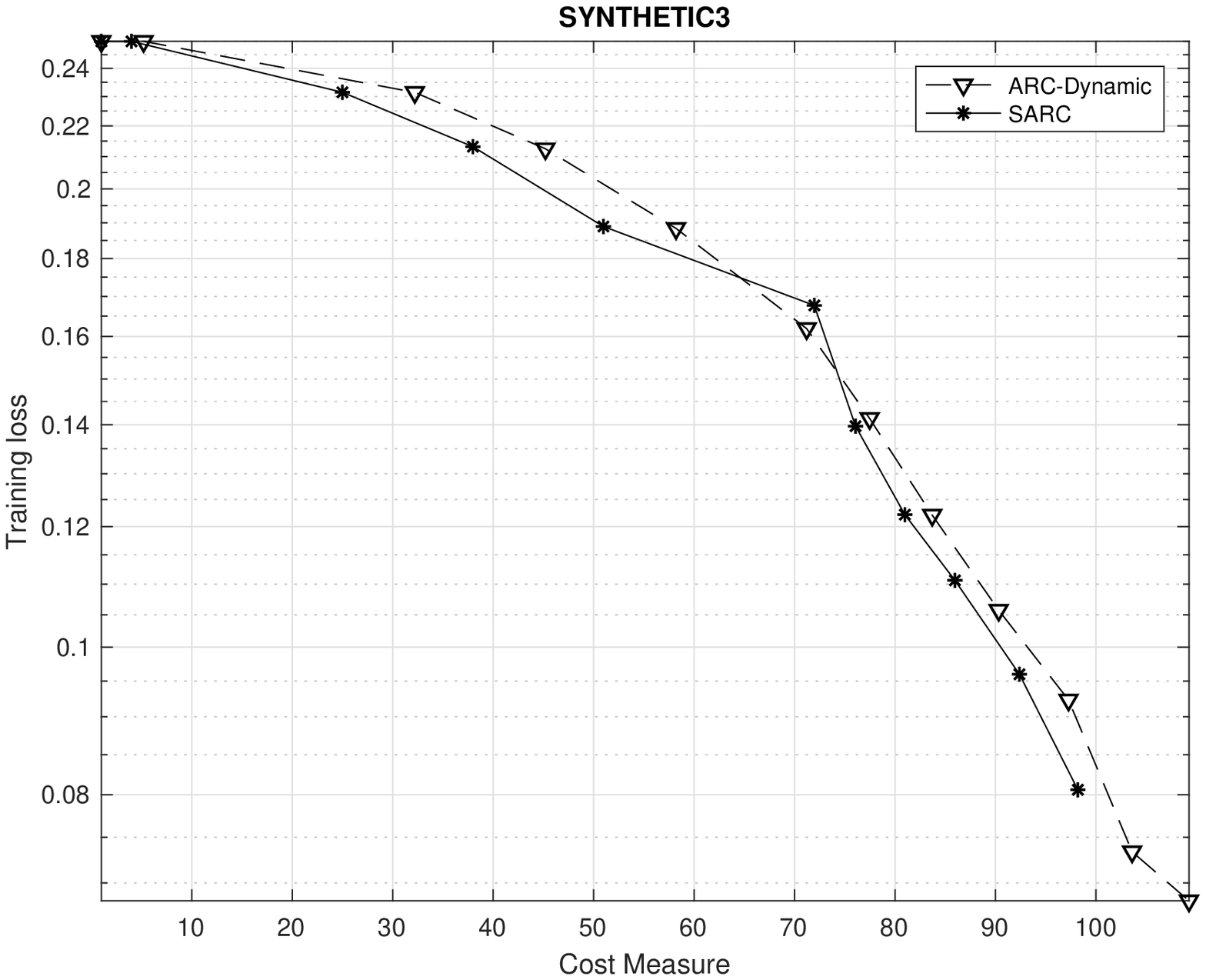}
\includegraphics[width=%
0.49\textwidth]{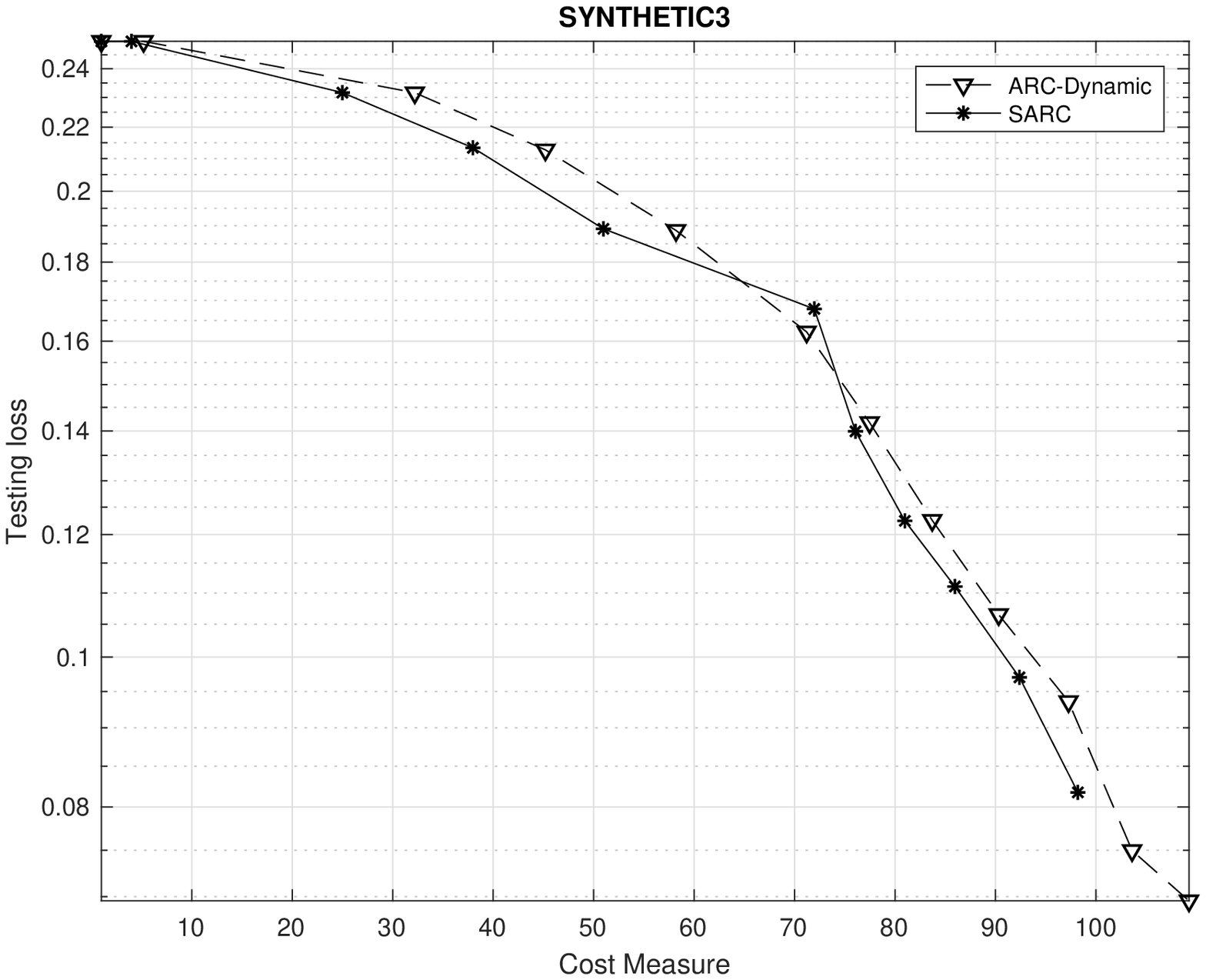}
\includegraphics[width=%
0.49\textwidth]{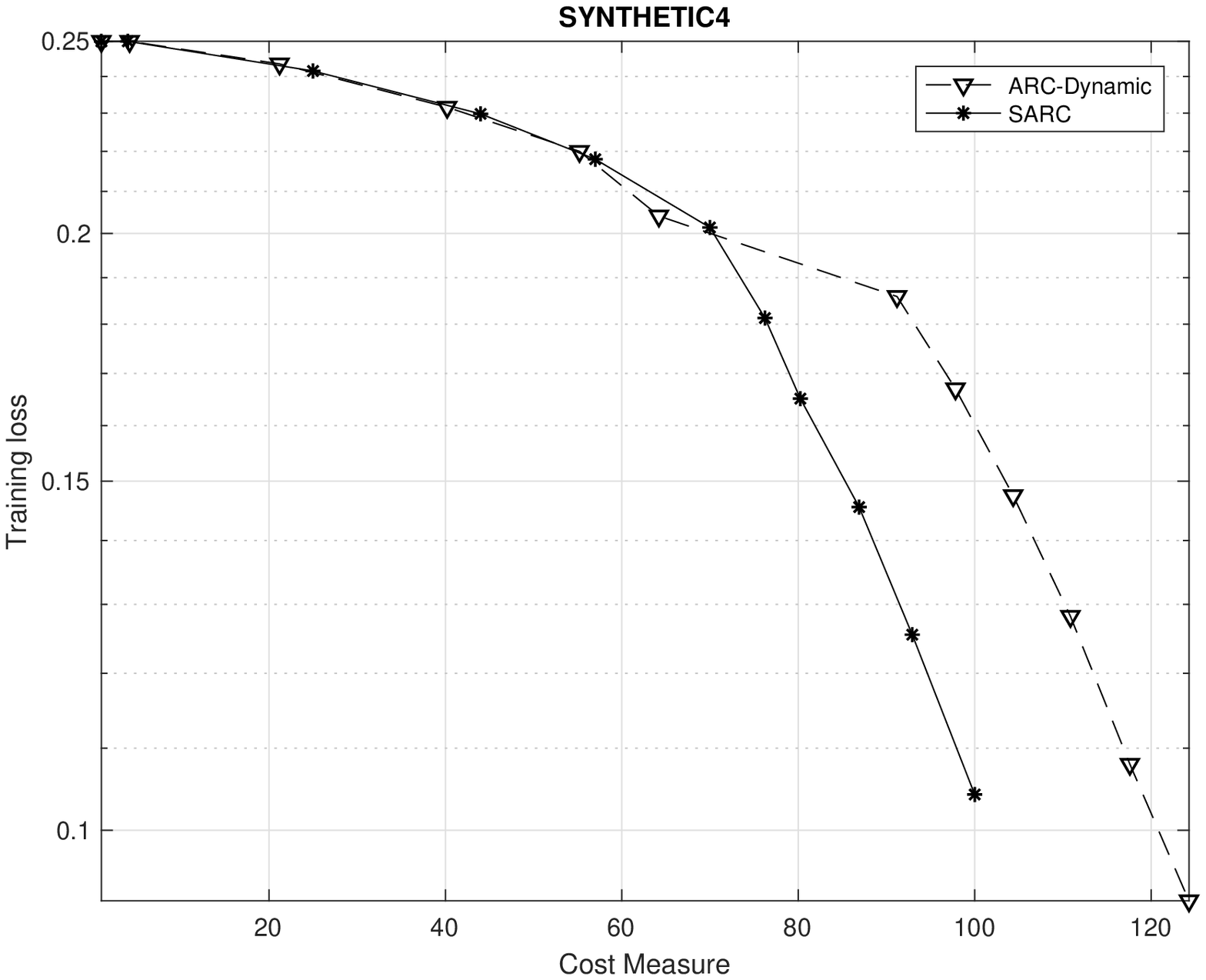}
\includegraphics[width=%
0.49\textwidth]{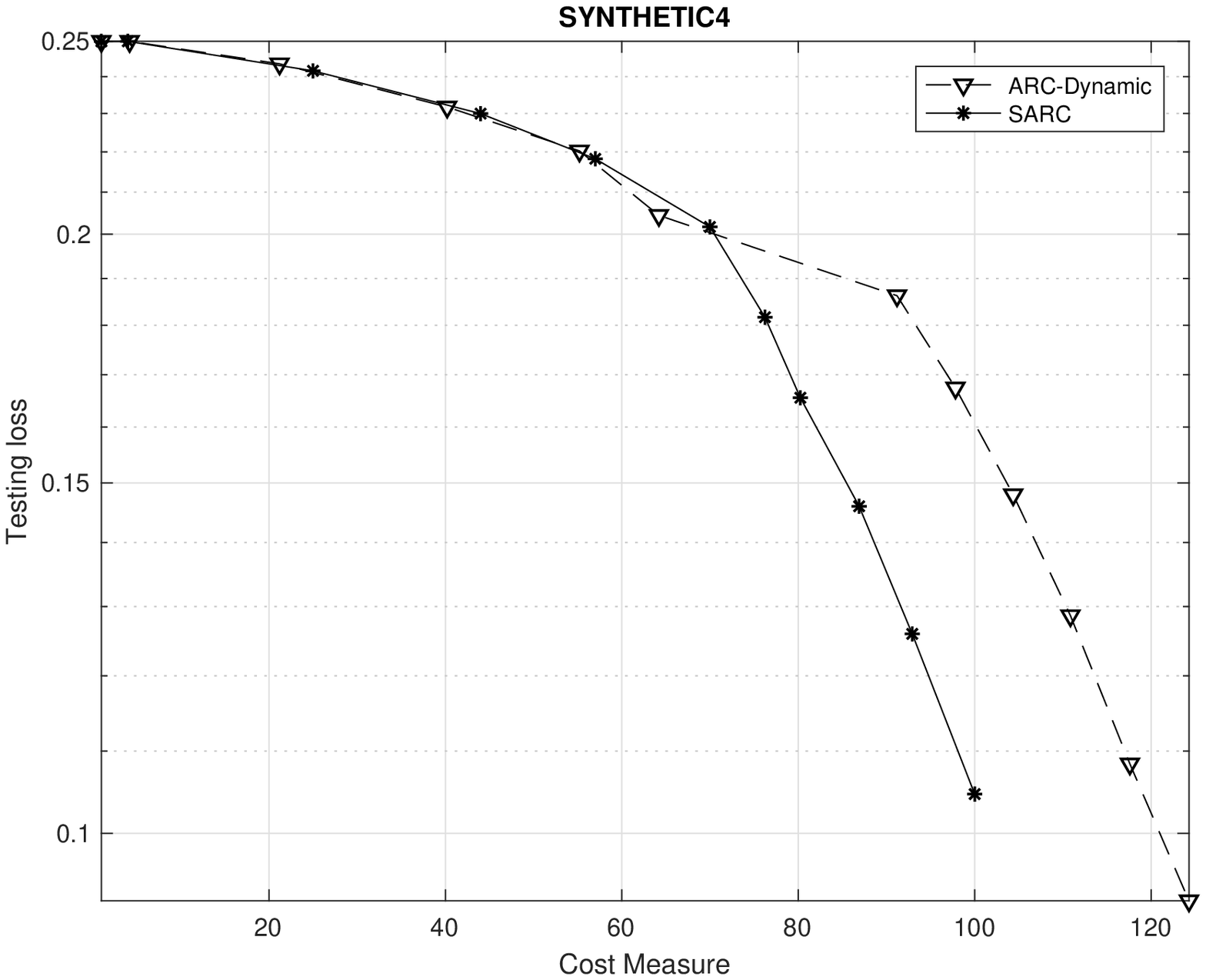}
\caption{Synthetic datasets. Comparison of \textit{SARC} (continuous line with asteriks) and \textit{ARC-Dynamic} (dashed line with triangles) against the considered cost measure CM. Each row corresponds to a different synthetic dataset. Training loss (left) and testing loss (right) against CM with logarithmic scale on the $y$ axis.}
\label{Perf1KL}
\end{figure}

\begin{figure}[h]
\centering
\includegraphics[width=%
0.49\textwidth]{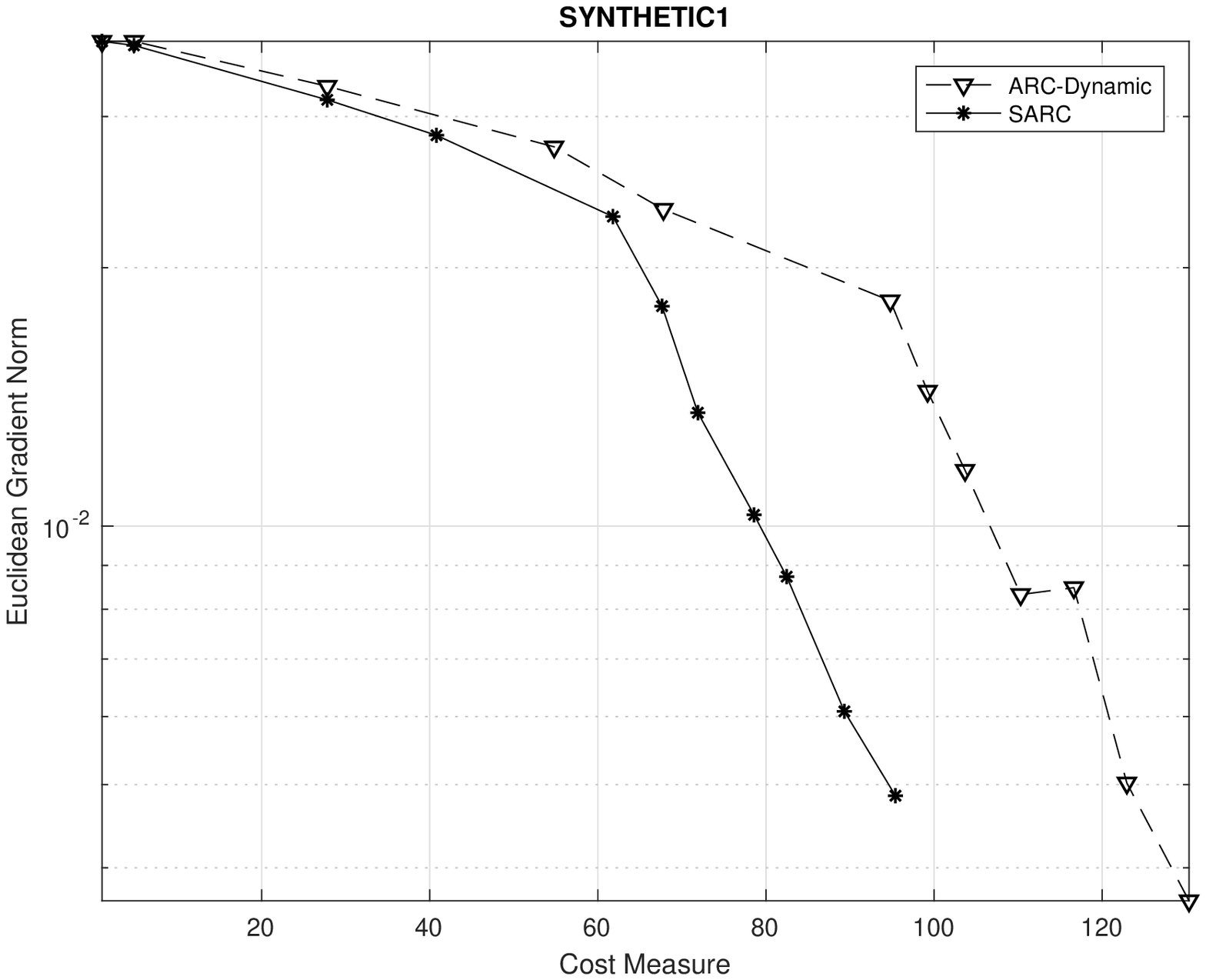}
\includegraphics[width=%
0.49\textwidth]{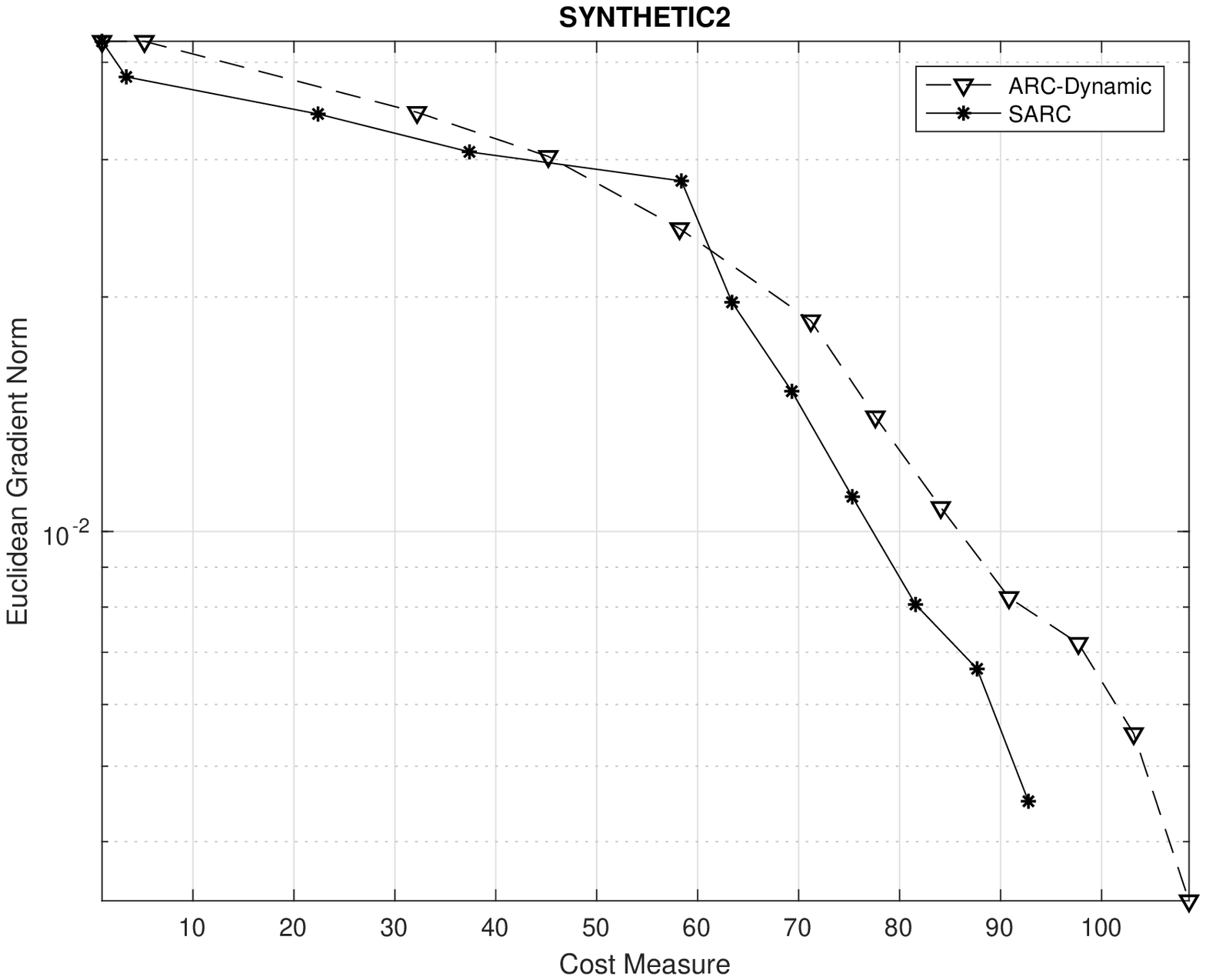}
\includegraphics[width=%
0.49\textwidth]{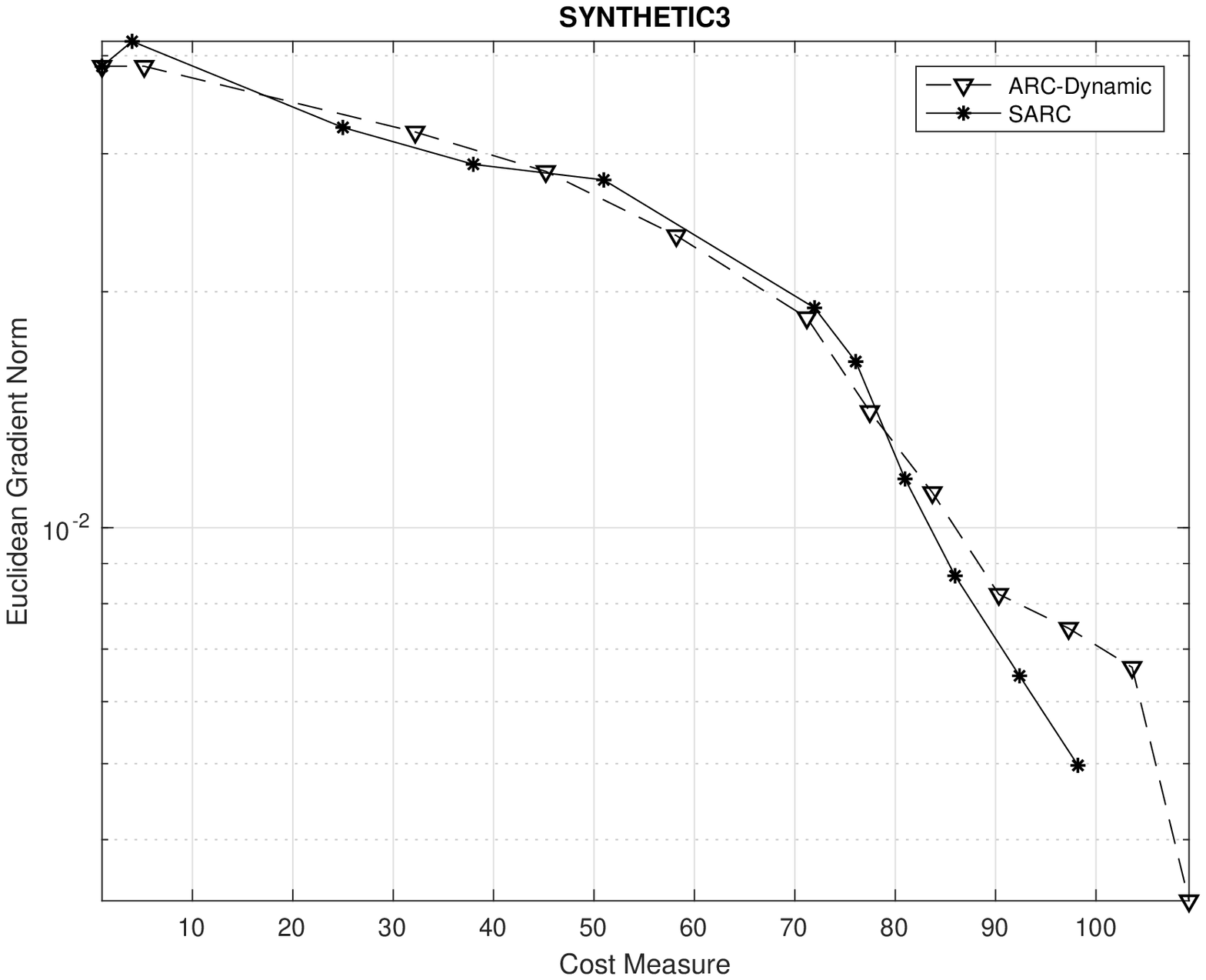}
\includegraphics[width=%
0.49\textwidth]{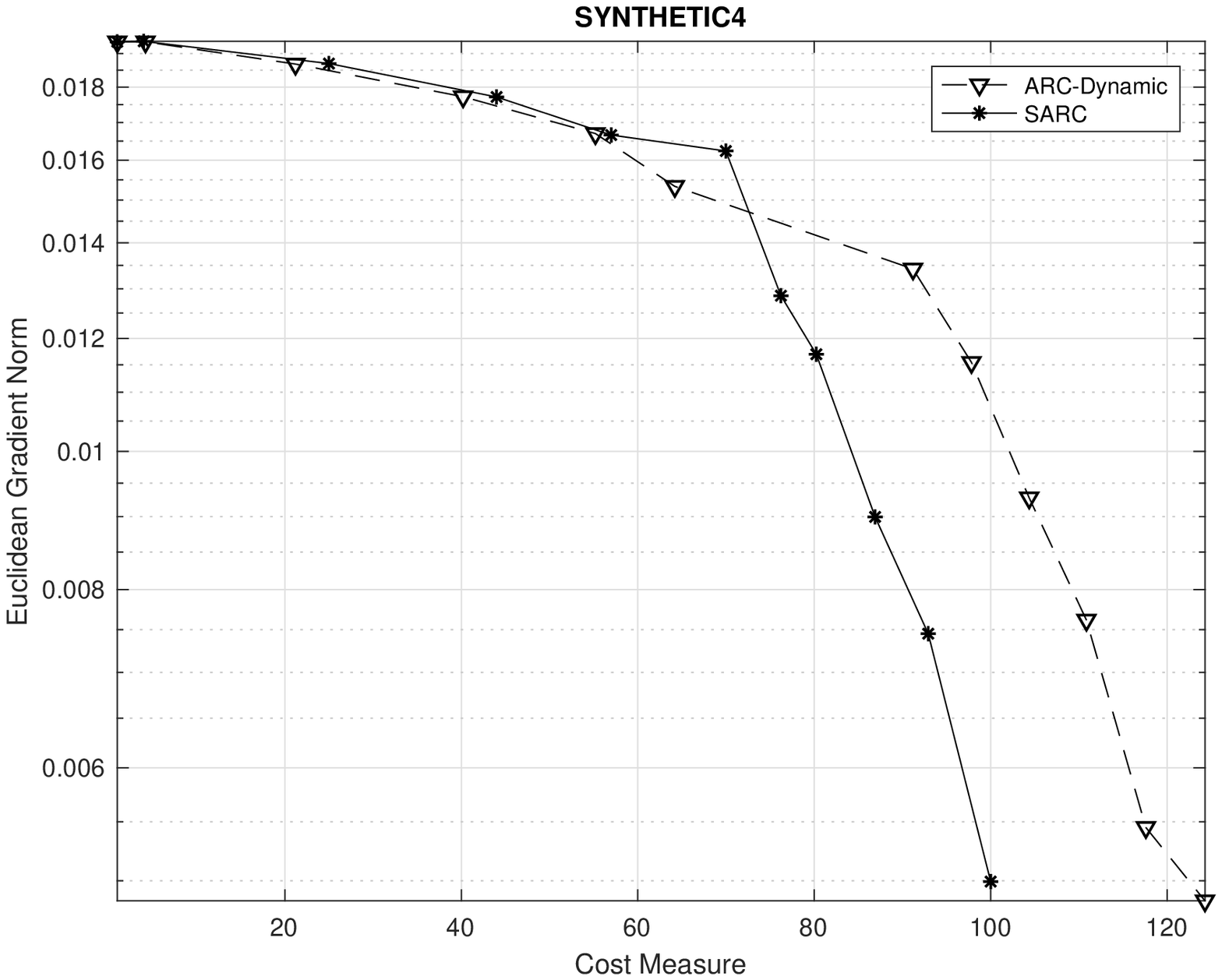}
\caption{Synthetic datasets. Euclidean norm of the gradient against CM (training set) with logarithmic scale on the $y$ axis. \textit{SARC} (continuous line with asteriks), \textit{ARC-Dynamic} (dashed line with triangles).
}
\label{GDsyntehticdata}
\end{figure}

\noindent
In all cases, Figure \ref{Perf1KL} shows the savings gained by $SARC$ in terms of the overall computational cost, as well as the improvements in the training phase and the testing accuracy under the same cost measure. More in general, we stress that second order methods show their strength  on these ill-conditioned datasets since all the tested procedures  manage to reduce the norm of the gradient and reach high accuracies in the classification rate. 
Even if we believe that reporting binary classifications accuracy obtained by each of the considered methods at termination is relevant in itself, we remark that the higher accuracy obtained at termination by \textit{ARC-Dynamic} (see Table \ref{BinAsyntheticdataset}) is just due to the fact the $SARC$ stops earlier. This should not be confused with a better performance of \textit{ARC-Dynamic}, since Figure \ref{Perf1KL} highlights that, along all datasets, when $SARC$ stops its testing loss is sensibly below the corresponding one performed by \textit{ARC-Dynamic} at the same CMT value. 

In Figure \ref{SampleSize}, we finally analyse the adaptive choices of the sample sizes  $\mathcal{D}_{j,k}$, $j\in\{1,2\}$, in \eqref{sizeD}.  As expected, the two strategies are more or less comparable when selecting the sample sizes for Hessian approximations, while the number of samples used to compute gradient approximations by $SARC$ oscillates across all iterations, always remaining far below the full sample size. In so doing, we outline that too small values of $\tau_0$ seem to have a bad influence on the performance of $SARC$, while as $\tau_0$ increases 
 it generally produces frequent saving in the CMT, once that it is above  a certain threshold value. In support of this observation, we report in Figure \ref{Figtau0} the variation of CMT against $\tau_0$  on Synthetic1 and Synthetic4.
We finally notice that, except for a few iterations at the first stage of the iterative process, the sample size for Hessian   approximation is lower than that used for  gradient approximation. This is in line with   the theory
as the  gradient is eventually required to be more accurate than the Hessian. In fact, the error in gradient approximation has to be of the order of $\|s_k\|^2$, while that in Hessian approximation has to be of the order of $\|s_k\|$,
see Lemma \ref{Lemmagk} and \ref{LemmaCk}.

\noindent
\begin{figure}[h]
\centering
\includegraphics[width=%
0.49\textwidth]{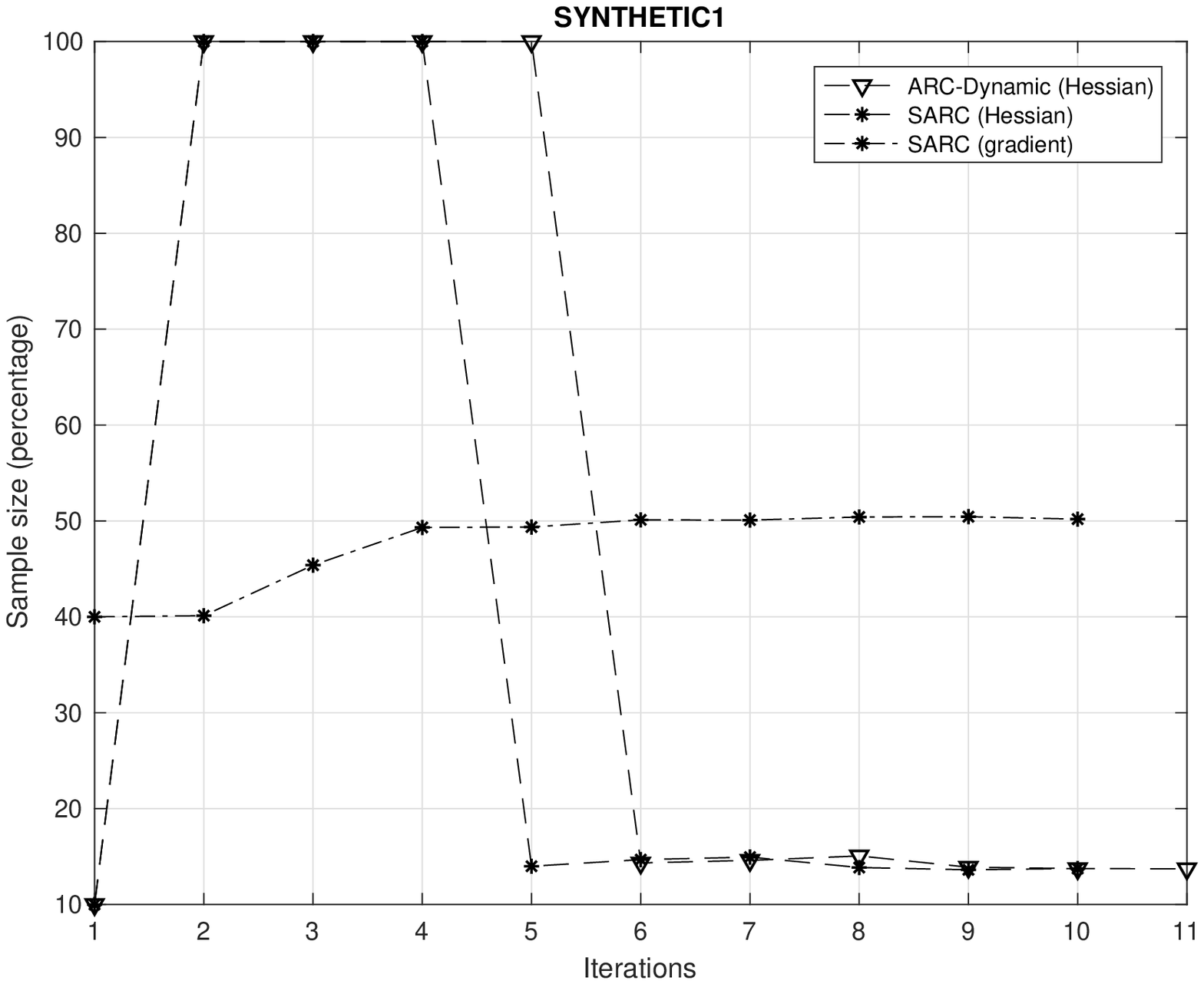}
\includegraphics[width=%
0.49\textwidth]{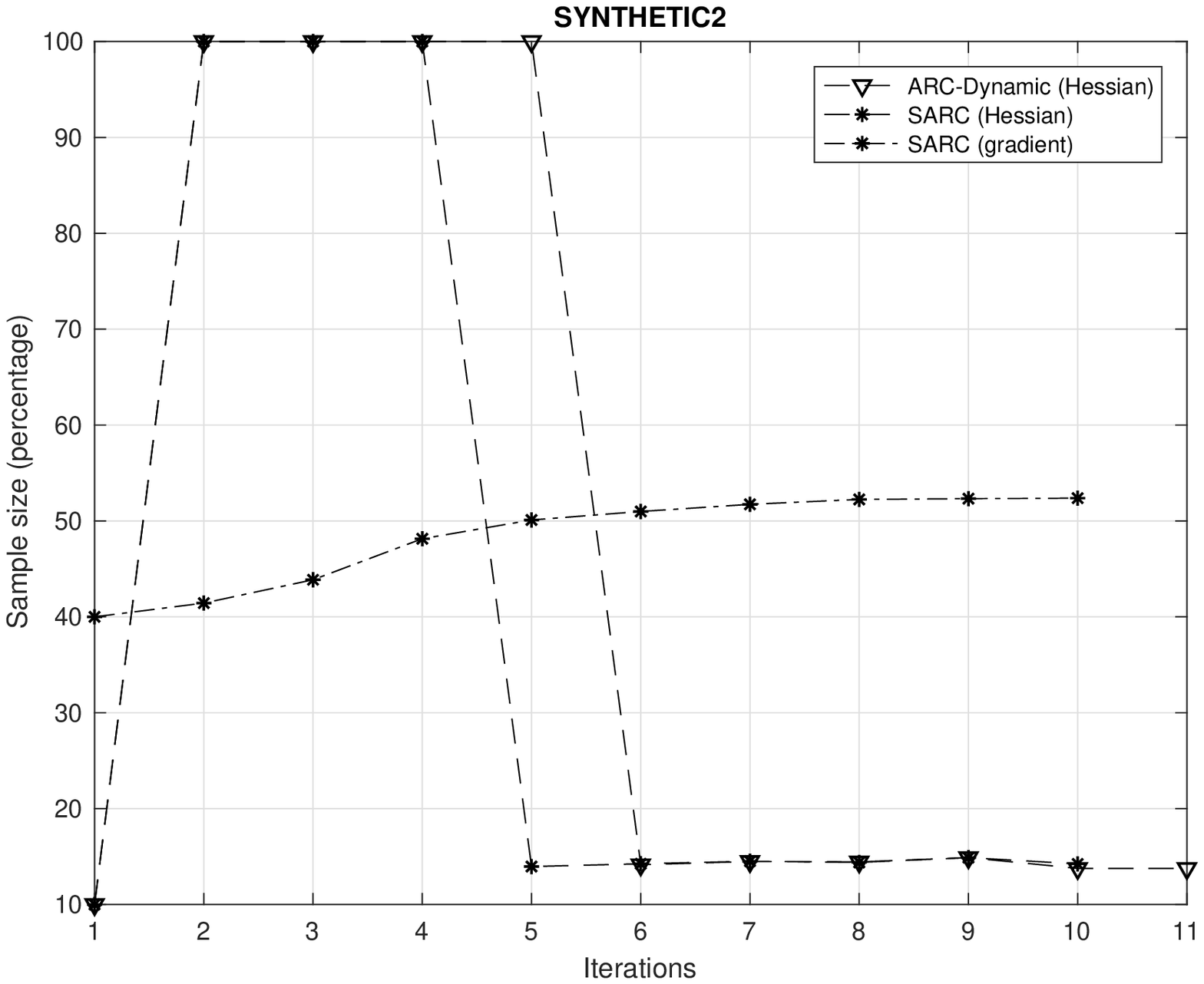}
\includegraphics[width=%
0.49\textwidth]{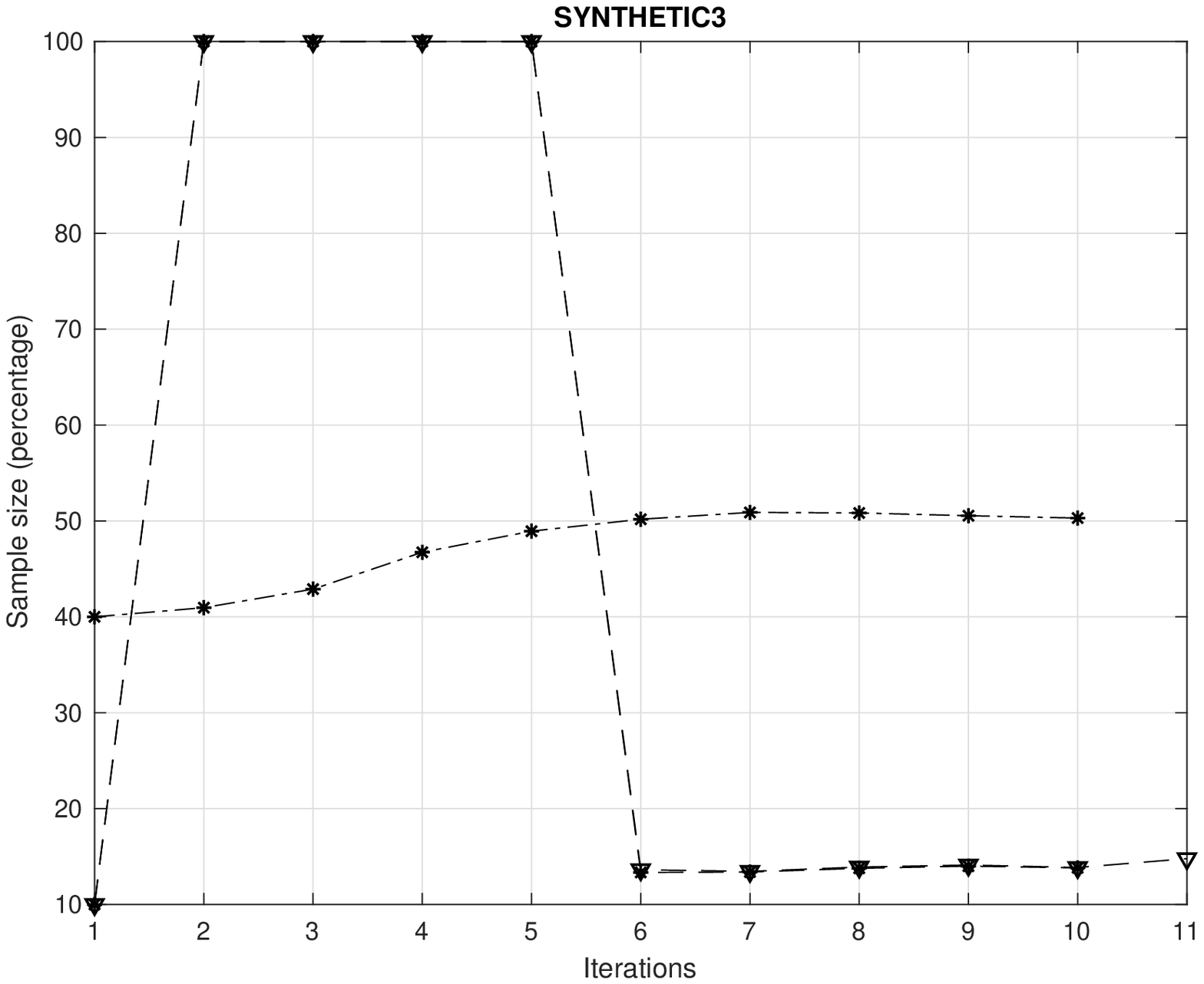}
\includegraphics[width=%
0.49\textwidth]{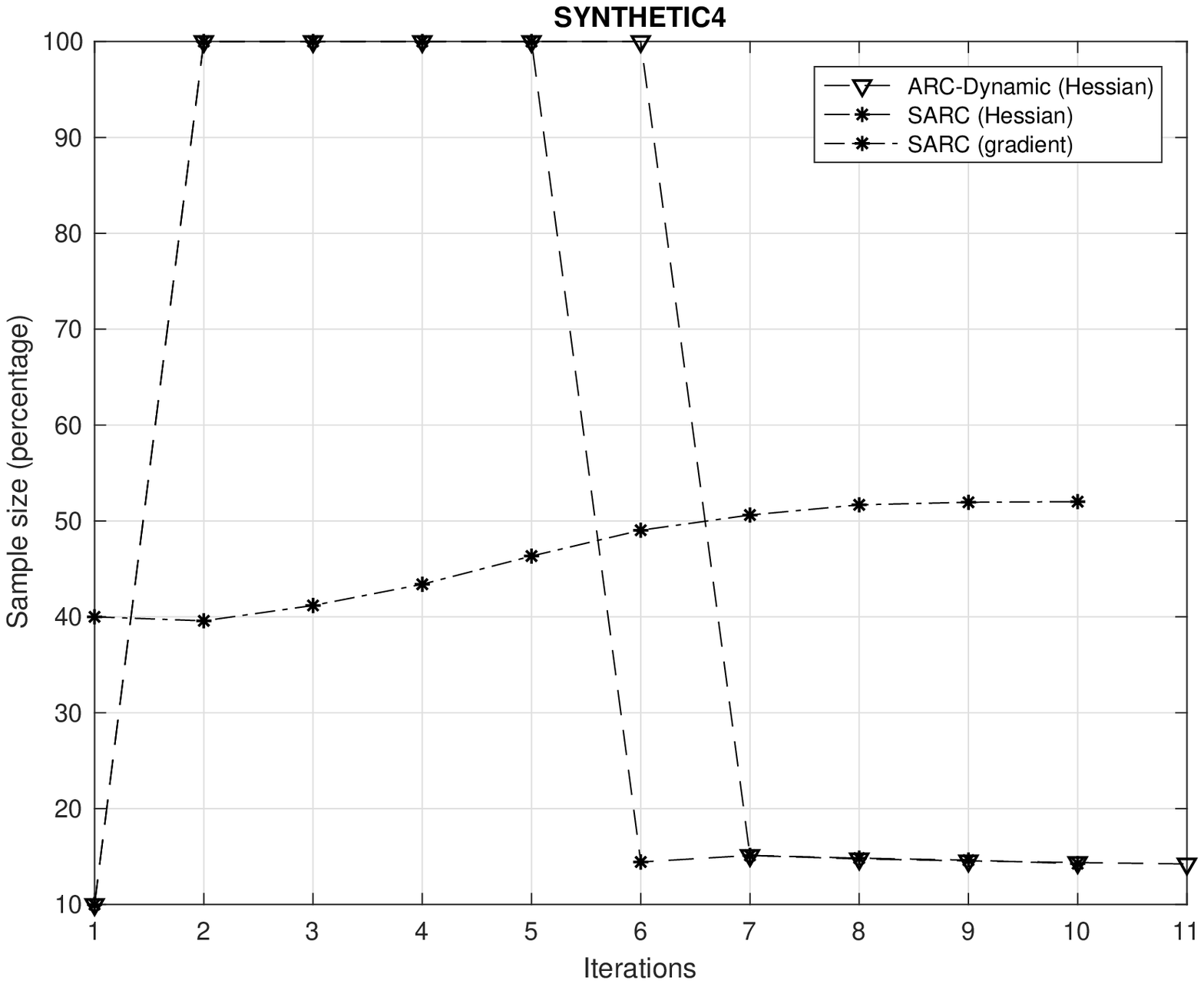}
\caption{Synthetic datasets. Sample size for Hessian approximations employed by \textit{ARC-Dynamic} (dashed line with triangles) and $SARC$ (dashed line with asteriks), together with the sample size for gradient approximations considered by $SARC$ (dotted dashed line with asteriks) against iterations.}
\label{SampleSize}
\end{figure}

\noindent
\begin{figure}[h]
\centering
\includegraphics[width=%
0.49\textwidth]{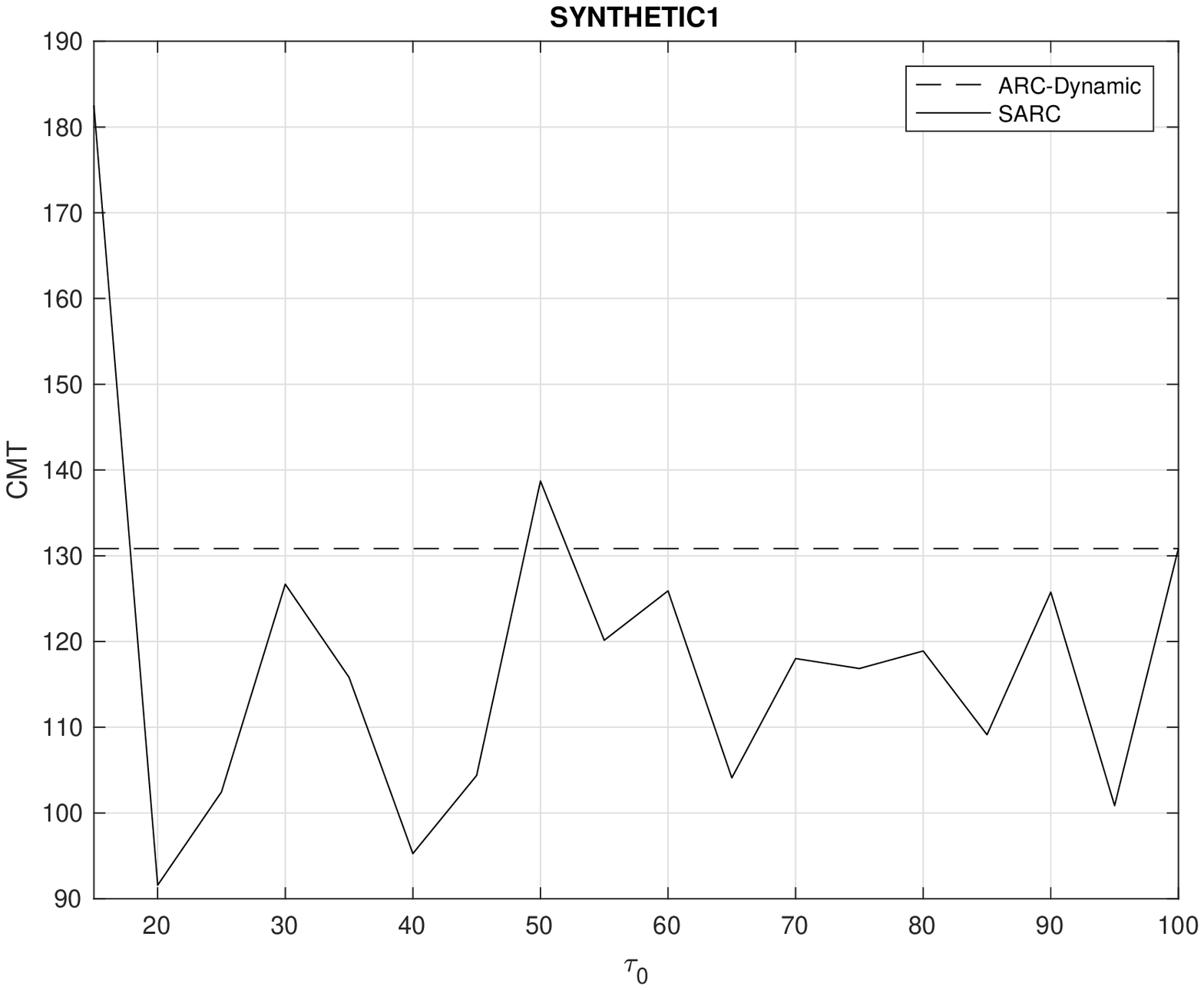}
\includegraphics[width=%
0.49\textwidth]{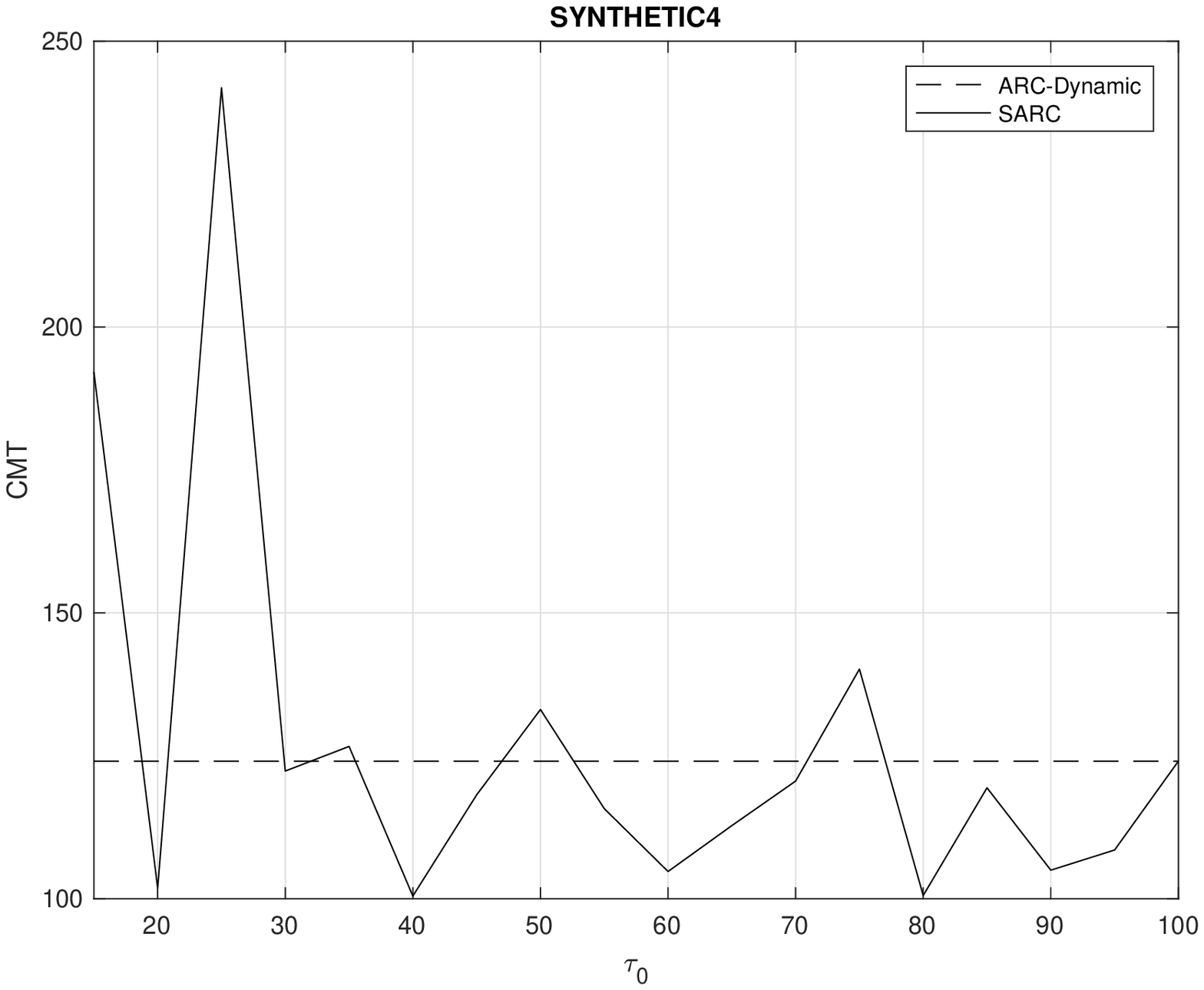}
\caption{Cost Measure at Termination (CMT) against $\tau_0$ among $SARC$ (continuous line) and \textit{ARC-Dynamic} (dashed line) on Synthetic1 and Synthetic4.}
\label{Figtau0}
\end{figure}

\section{Conclusion and perspectives} We have proposed the stochastic analysis of the process generated by an ARC algorithm for solving unconstrained, nonconvex, optimisation problems under inexact derivatives information. The algorithm is an extension of the one in \cite{IMA}, since it employs approximated evaluations of the gradient with the main feature of mantaining the dynamic rule for building Hessian approximations, introduced and numerically tested in \cite{IMA}. This kind of accuracy requirement is always reliable and computable when an approximation of the exact Hessian is needed by the scheme and, in contrast to other strategies such that the one in \cite{CartSche17}, does not require the inclusion of additional inner loops to be satisfied. With respect to the framework in \cite{IMA}, where in the finite-sum setting optimal complexity is restored with high probability, we have here provided properties of the method when the adaptive accuracy requirements of the derivatives involved in the model definition are not accomplished, with a view to search for the number of expected steps that the process takes to reach the prescribed accuracy level. The stochastic analysis is thus performed exploiting the theoretical framework given in \cite{CartSche17}, showing that the expected complexity bound matches the worst-case optimal complexity of the ARC framework. The possible lack of accuracy of the model has just the effect of scaling the optimal complexity we would derive from the deterministic analysis of the framework (see, e.g., \cite[Theorem~4.2]{IMA}), by a factor which depends on the probability $p$ of the model being sufficiently accurate. Numerical results confirm the theoretical achievements and highlight the improvements of the novel strategy on the computational cost  in most of the tests  with no worsening of the binary classification accuracy. This paper does not cover the case of noisy functions (\cite{PaquSche18, Chen15, ChenMeniSche18}), as well as the second-order complexity analysis. The stochastic second-order complexity analysis of ARC methods with derivatives and function estimates will be a challenging line of investigation for future work. Concerning the latter point, we remark that a recent advance in \cite{STR2}, based on properties of supermartingales, has tackled with the second-order convergence rate analysis of a stochastic trust-region method. 
  \vskip 5 pt
 \noindent
 {\bf Funding}: the authors are member of the INdAM Research Group GNCS and partially supported by INdAM-GNCS  through Progetti di Ricerca 2019. 

  \vskip 5 pt
 \noindent
  {\bf Acknowledgements.} The authors  dedicate  this paper, in honor of his 70th birthday,  to  Alfredo Iusem. Thanks are due  to Coralia Cartis, Benedetta Morini and Philippe Toint for fruitful discussion on stochastic complexity analysis and to two anonymous referees whose comments significantly improved the presentation of this paper.

\end{document}